\documentclass[12pt]{amsart}
\usepackage{amsmath,amsthm,mathrsfs,amsfonts,amssymb,color}

\newtheorem{thm}{Theorem}[section]
\newtheorem{lem}[thm]{Lemma}

\newtheorem{pro}[thm]{Proposition}
\newtheorem{defn}[thm]{Definition}

\def\binfty {\dot{B}_\infty^{\beta,\infty}}

\numberwithin{equation}{section}
\allowdisplaybreaks
\begin{document}

\title[Lipschitz and Triebel--Lizorkin spaces in Dunkl setting]
{Lipschitz and Triebel--Lizorkin spaces, commutators\\ in Dunkl setting}

\thanks{JL's research supported by ARC DP 220100285 and NNSF (grant no. 12171221 and 12071197). 
MYL's research supported by NSTC 112-2115-M-008-001-MY2.
BDW’s research is supported in part by National Science Foundation Grants DMS \#1800057, \#2054863, and \#20000510 and Australian Research Council – DP 220100285.
}

\author{Yongsheng Han}
\address{Yongsheng Han, Department of Mathematics, Auburn University, Auburn, Alabama, USA}
\email{hanyong@auburn.edu}

\author{Ming-Yi Lee}
\address{Ming-Yi Lee, Department of Mathematics, National Central University, Taiwan, Republic of China}
\email{mylee@math.ncu.edu.tw}

\author{Ji Li}
\address{Ji Li, Department of Mathematics, Macquarie University, Sydney}
\email{ji.li@mq.edu.au}

\author{Brett D. Wick}
\address{Brett D. Wick, Department of Mathematics\\
         Washington University - St. Louis\\
         St. Louis, MO 63130-4899 USA
         }
\email{wick@math.wustl.edu}

\begin{abstract}
We first study the Lipschitz spaces $\Lambda_{d}^\beta$ associated with the Dunkl metric, $\beta\in(0,1)$, and prove that it is a proper subspace of the classical  Lipschitz spaces $\Lambda^\beta$ on $\mathbb R^N$, as the Dunkl metric and the Euclidean metric are non-equivalent. Next, we further show that the Lipschitz spaces $\Lambda^\beta$ connects to the
Triebel--Lizorkin spaces $\dot{ F}^{\alpha,q}_{p,{\rm D}}$ associated with the Dunkl Laplacian $\triangle_{\rm D}$ in $\mathbb R^ N $ and to  the commutators of the Dunkl Riesz transform and the  fractional Dunkl Laplacian $\triangle_{\rm D}^{-\alpha/2}$, $0<\alpha<\textbf{N}$ (the homogeneous dimension for Dunkl measure), which is represented via the functional calculus of the Dunkl heat semigroup $e^{-t\triangle_{\rm D}}$.  The key steps in this paper are a finer decomposition of the underlying space via Dunkl metric and Euclidean metric to bypass the use of Fourier analysis, and a discrete weak-type Calder\'on reproducing formula in these new Triebel--Lizorkin spaces $\dot{ F}^{\alpha,q}_{p,{\rm D}}$.

\end{abstract}


\maketitle
\section {Introduction and statement of main results}\label{sec-Intro}
\subsection{Background}
It is well-known that the Lipschitz spaces and Triebel--Lizorkin spaces play an important role in modern harmonic analysis and partial differential equations. 

In this paper, we establish a theory of Lipschitz spaces in the Dunkl setting
and prove that the commutators of Dunkl Riesz transforms  characterize these spaces. Important tools are the reproducing formula built on Dunkl--Triebel--Lizorkin spaces and their duality, which is of independent interest for other  function spaces and the related differential equations in Dunkl setting.

It is well known that group structures enter in a decisive way in harmonic analysis. Dunkl theory is associated with a finite reflection group in $\mathbb R^N$. It started with the papers
\cite{D1, D2, D3} and was further developed in the papers \cite{AAS, AS,  DW, DK, Jeu, LL, R, S, TX}. In \cite{ADH}, the authors extended the classical Hardy $H^1$ to the rational Dunkl setting. 
To be more precise, in $\mathbb R^N$, the reflection $\sigma_\alpha$ with respect to the hyperplane $\alpha_\bot$ orthogonal to a nonzero vector $\alpha$ is given by
$ \sigma_\alpha(x) = x- 2 { \langle x,\alpha\rangle} \alpha\ \|\alpha\|^{-2}.$

A finite set $R\subset \mathbb R^N\backslash\{0\}$ is called a \textit{root system} if $\sigma_\alpha(R) = R$.
Let $R$ be a root system in $\mathbb{R}^{N}$ normalized so that $\langle\alpha,\alpha\rangle=2$ for $\alpha\in R$ and with $R_+$ a fixed positive subsystem, and $G$ the finite reflection group generated by the reflections $\sigma_\alpha$ ($\alpha\in R$), where $\sigma_{\alpha}(x)=x-\langle\alpha,x\rangle\alpha$ for $x\in\mathbb{R}^{N}$.  Corresponding to this reflection group, we denote by $\mathcal{O}(x)$ 
the $G$-orbit of a point
$ x\in\mathbb{R}^{N}$.  There is a natural metric between two G-orbits $\mathcal{O}(x)$ and $\mathcal{O}(y)$, given by
$$d(x,y):=\min\limits_{\sigma\in G}\|x-\sigma(y)\|.$$
It is clear that $d(x,y)\leq \|x-y\|$. And it is possible that for certain $x,y\in\mathbb R^N$, $d(x,y)=0$ while $\|x-y\|>0$.

The closures of connected components of $\{x\in \mathbb R^N : \langle x,\alpha\rangle\ne 0$ for all $\alpha\in R\}$
are called (closed) \textsl{Weyl chambers}. We remark that $\|x-y\|=d(x,y)$ if and only if $x, y\in \mathbb R^N$
belong to the same closed Weyl chamber, see \cite{DH3}.

A \textsl{multiplicity function} $\kappa$ defined on $R$ (invariant
under $G$) is fixed throughout this paper. Let
\begin{eqnarray*}
    d\omega(x)=\prod_{\alpha\in R}|\langle\alpha,x\rangle|^{\kappa(\alpha)}dx
\end{eqnarray*}
be the associated measure in $\mathbb{R}^{N},$ see \cite{ADH},
where, here and subsequently, $dx$ stands for the Lebesgue measure
in $\mathbb{R}^{N}$. We denote by $\textbf{N} = N
+\sum\limits_{\alpha\in R}\kappa(\alpha)$ the homogeneous dimension. The measure $d\omega$ satisfies 
\begin{equation}\label{rd1.1}
C^{-1}\bigg(\frac{r_2}{r_1}\bigg)^{N}\leqslant
\frac{w(B(x,r_2))}{w(B(x,r_1))}\leqslant C\bigg(\frac{r_2}{r_1}\bigg)^{\textbf N} \qquad\text{for}\quad
0<r_1<r_2,
\end{equation}
which implies that there is a constant $C > 0$ such that
for all $ x\in \mathbb{R}^{N}, r>0$ and $\lambda \geqslant 1,$
\begin{eqnarray}\label{rd}
C^{-1}\lambda^{N}\omega(B(x, r)) \leqslant \omega(B(x, \lambda
r))\leqslant C\lambda^{\textbf N}\omega(B(x, r)).
\end{eqnarray}

The Dunkl (differential) operators $T_j$ are defined by
$$T_jf(x)=\partial_j f(x) + \sum\limits_{\alpha\in R^+}\frac{\kappa(\alpha)}{2}\langle\alpha, e_j\rangle\frac{f(x)-f(\sigma_\alpha(x))}{\langle \alpha, x\rangle},$$
where $e_1, \ldots, e_N$ are the standard unit vectors of $\mathbb
R^N.$
The Dunkl Laplacian is then defined as
$\triangle_{\rm D}=\sum\limits^{N}_{j=1} T^2_j,$ which is equivalent to
$$\triangle_{\rm D} f(x)=\triangle_{\mathbb R^N} f(x) + \sum\limits_{\alpha \in R}\kappa(\alpha)\Big(\frac{\partial_{\alpha}f(x)}{\langle\alpha,
	x\rangle}
-\frac{f(x)-f(\sigma_{\alpha}(x))}{\langle\alpha,
	x\rangle^2}\Big).$$
Here $\triangle_{\mathbb R^N}$ is the standard Euclidean Laplacian.

The operator $\triangle_{\rm D}$ is self-adjoint on
$L^2(\mathbb R^N,\omega)$, see \cite{ADH}, and generates the heat semigroup
$$H_tf(x)=e^{t\triangle_{\rm D}}f(x)=\int_{\mathbb R^N}h_t(x,y)f(y)d\omega(y),$$
where the heat kernel $h_t(x,y)$ is a $C^\infty$ function for all
$t>0, x, y\in \mathbb R^N$ and satisfies $h_t(x, y)=h_t(y, x)>0$ and
$\int_{\mathbb R^N} h_t(x, y) d\omega(y)=1.$ The size and regularity estimate of 
$h_t(x, y)$ were established in \cite{ADH}.

\noindent{\bf Theorem A (\cite[Theorem 4.1]{ADH}).}
{\it  
   There are constants $C,c>0$ such that
  
{\rm (a)}  for every $t>0$ and for every $x,y\in \mathbb R^N$, $${1\over C{\min\{\omega(B(x,\sqrt t)), \omega(B(y,\sqrt t))\}}}e^{-c\|x-y\|^2/t}\leq |h_t(x,y)|\le CV(x,y,\sqrt t)^{-1}e^{-cd(x,y)^2/t},$$
 where $V(x,y,r):=\max\{\omega(B(x,r)),\omega(B(y,r))\}$;

{\rm  (b)} for every $t>0$ and for every $x,y,y'\in \mathbb R^N$ such that $\|y-y'\|<\sqrt t$,
  $$|h_t(x,y)-h(x,y')|\le C\bigg(\frac{\|y-y'\|}{\sqrt t}\bigg)V(x,y,\sqrt t)^{-1}e^{-cd(x,y)^2/t}.$$

}
%\end{thm}

In \cite{DH2}, they improve the upper and lower bounds for Dunkl heat kernel : for all $c_l>\frac14$ and $0<c_u<\frac14$
there are constants $C_l, C_u>0$ such that 
$$C_l\omega(B(x,\sqrt t))^{-1}e^{-c_l\frac{d(x,y)^2}t}\Lambda(x,y,t)\leq h_t(x,y)\le C_u\omega(B(x,\sqrt t))^{-1}e^{-c_u\frac{d(x,y)^2}t}\Lambda(x,y,t)$$
where $\Lambda(x,y,t)$ can be expressed by means of some rational functions of $\|x-\sigma(y)\|/\sqrt t$.

\subsection{Statement of main results}
\subsubsection{Part I: Lipschitz Spaces}

We recall the standard Lipschitz space on $\mathbb R^N$ and introduce a new  Lipschitz space associated with the Dunkl metric $d$.
\begin{defn}
For $\beta\in (0,1)$, the Lipschitz space $\Lambda^\beta$ is the set of functions $f$ defined on $\mathbb R^N$ such that 
\begin{align}
\|f\|_{\Lambda^\beta}:=&\sup_{x\ne y} \frac{|f(x)-f(y)|}{\|x-y\|^\beta}<\infty.\label{e4.0}
\end{align}
We now introduce the Lipschitz space $\Lambda^\beta_d$ associated with the Dunkl metric $d$ as follows.
For $\beta\in (0,1)$,  $\Lambda^\beta_d$ is the set of functions $f$ defined on $\mathbb R^N$ such that 
\begin{align}\label{e4.0d}
\|f\|_{\Lambda^\beta_d}:=\sup_{d(x,y)\ne0} \frac{|f(x)-f(y)|}{d(x,y)^\beta}<\infty.
\end{align}

\end{defn}    

We will show  equivalent definitions of these two spaces which will be useful in applications to commutator operators.
\begin{thm}\label{lip and lipd}
For $\beta\in (0,1)$ and $q\in [1, \infty)$, the following are equivalent to \eqref{e4.0}:
\begin{align}
&\sup_{B\subseteq \mathbb R^ N }\frac 1{\ell(B)^\beta\omega(B)}\int_B |f(x)-f_
B|d\omega(x);\label{e4.1} \\
&\sup_{B\subseteq \mathbb R^ N }\frac 1{\ell(B)^{\beta}}\bigg(\frac1{\omega(B)}\int_B |f(x)-f_B|^q d\omega(x)\bigg)^{\frac1q},\label{e4.2} 
\end{align}
where the sup is taken over all balls $B$ in $\mathbb{R}^N$ and $f_B = \frac 1{\omega(B)}\int_B f(x )d\omega(x)$.

Moreover, the following are equivalent to  \eqref{e4.0d}:
\begin{align}
&\sup_{B\subseteq \mathbb R^ N }\frac 1{\ell(B)^{\beta}\omega(\mathcal O(B))}\int_{\mathcal O(B)} |f(x)-f_{\mathcal O(B)}| d\omega(x); \label{e4.1d}\\
&\sup_{B\subseteq \mathbb R^ N }\frac 1{\ell(B)^{\beta   }}\bigg(\frac1{\omega(\mathcal O(B))}\int_{\mathcal O(B)} |f(x)-f_{\mathcal O(B)}|^q d\omega(x)\bigg)^{\frac1q},\label{e4.2d}
\end{align}
where the sup is taken over all balls $B$ in $\mathbb{R}^N$ and $f_{\mathcal O(B)} = \frac 1{\omega(\mathcal O(B))}\int_{\mathcal O(B)} f(x )d\omega(x)$.

\end{thm}

Finally, we have the following containment relation.
\begin{thm}\label{lip comp}
For $\beta\in (0,1)$, $ \Lambda^\beta_d\subsetneq \Lambda^\beta $.
\end{thm}

%To obtain the discrete weak-type Calder\'on reproducing formula in Dunkl--Triebel--Lizorkin spaces, we 
\subsubsection{Part II: The Dunkl--Triebel--Lizorkin Space and Commutators}

Here and throughout, we denote by $L^p(\mathbb R^N,\omega)$, $1\leq p<\infty$, the set of all measurable functions $f$ on $\mathbb R^N$ such that  
$\|f\|_{L^p_\omega}:=\big\{\int_{\mathbb R^N} |f(x)|^p d\omega(x)\big\}^{1/p}<\infty.$

We now introduce a new test function space as follows and then introduce the Dunkl--Triebel--Lizorkin spaces.
\begin{defn}\rm
Let ${\mathfrak R}>1$ be a fixed constant. Let $q_{{\mathfrak R}^k}$ and $\psi_{{\mathfrak R}^k}$ be given as in Theorem \ref{CRF}. 
For $|\alpha|<1<{\mathfrak R}$, $1< p<\infty$ and $1\leq q\leq\infty$, define
$\dot{\mathcal F}^{\alpha,q}_{p,{\rm D}}=\{f\in L^2(\Bbb R^N, \omega) : \|f\|_{\dot{\mathcal F}^{\alpha,q}_{p,{\rm D}}}<\infty\},$
where
$$\|f\|_{\dot{\mathcal F}^{\alpha,q}_{p,{\rm D}}}
	=\begin{cases} \displaystyle \Big\|\Big\{\sum_{k\in \mathbb Z}\sum_{Q\in Q^k} \big({\mathfrak R}^{k\alpha}|q_{{\mathfrak R}^k}(f)(x_Q)|\big)^q\chi_Q\Big\}^{1/q}\Big\|_{L^p_\omega}\quad & \text{if}\ \ 1\le q<\infty \\
	\displaystyle \Big\|\sup_{k\in \mathbb Z}\sum_{Q\in Q^k} {\mathfrak R}^{k\alpha}|q_{{\mathfrak R}^k}(f)(x_Q)|\chi_Q\Big\|_{L^p_\omega} &\text{if}\ \ q=\infty. \end{cases}$$
Similarly, define	
$\dot{\mathcal F}^{\alpha,q}_{p,\psi}=\{f\in L^2(\Bbb R^N, \omega) : \|f\|_{\dot{\mathcal F}^{\alpha,q}_{p,\psi}}<\infty\},$
where
$$\|f\|_{\dot{\mathcal F}^{\alpha,q}_{p,\psi}}
	=\begin{cases} \displaystyle \Big\|\Big\{\sum_{k\in \mathbb Z}\sum_{Q\in Q^k} \big({\mathfrak R}^{k\alpha}|\psi_{{\mathfrak R}^k}(f)(x_Q)|\big)^q\chi_Q\Big\}^{1/q}\Big\|_{L^p_\omega}\quad & \text{if}\ \ 1\le q<\infty \\
	\displaystyle \Big\|\sup_{k\in \mathbb Z}\sum_{Q\in Q^k} {\mathfrak R}^{k\alpha}|\psi_{{\mathfrak R}^k}(f)(x_Q)|\chi
	_Q\Big\|_{L^p_\omega} &\text{if}\ \ q=\infty. \end{cases}$$
\end{defn}

The Dunkl--Triebel--Lizorkin space is then given by the following definition.
\begin{defn} \label{Triebel}
Let $|\alpha|<1$, $1< p<\infty$ and $1\leq q\leq\infty$.
    The Dunkl--Triebel--Lizorkin space ${\dot F}^{\alpha,q}_{p,{\rm D}}$ is defined by the collection of distributions $f\in (\dot{\mathcal F}^{-\alpha,q'}_{p',{\psi}})^\prime$ such that 
    \begin{align}\label{series f}
    f(x)=\sum\limits_{j=-\infty}^\infty\sum\limits_{Q\in Q^j}\omega(Q)
    \lambda_{Q}\psi_{Q}(x,x_{Q})
    \end{align}
    with $\Big\|\Big\{
    \sum\limits_{j=-\infty}^\infty\sum\limits_{Q\in Q^{j}}|\lambda_{Q}|^q\chi_{Q}(\cdot)
    \Big\}^{1/q}\Big\|_{L^p_\omega}<\infty,$ where the series converges in the distributional sense and $\psi_Q=\psi_j$ if $Q\in Q^j.$
If $f\in {\dot F}^{\alpha,q}_{p,{\rm D}},$ the norm of $f$ is defined by
    $$\|f\|_{{\dot F}^{\alpha,q}_{p,{\rm D}}}:=\inf \Bigg\{\Big\|\Big\{ \sum\limits_{j=-\infty}^\infty
    \sum\limits_{Q\in Q^{j}}|\lambda_{Q}|^q\chi_{Q} (\cdot)
    \Big\}^{1/q}\Big\|_{L^p_\omega}\Bigg\},$$ where the infimum is taken over all
    $f(x)$ in the form \eqref{series f}.
\end{defn}
We remark that if $\Big\|\Big\{
\sum\limits_{j=-\infty}^\infty\sum\limits_{Q\in Q^{j}}|\lambda_{Q}|^q\chi_{Q}
\Big\}^{1/q}\Big\|_{L^p_\omega}<\infty,$ then the series in \eqref{series f} defines a distribution in $(\dot{\mathcal F}^{-\alpha,q'}_{p',{\psi}})^\prime.$ 
We have the following characterization of the  Dunkl--Triebel--Lizorkin space ${\dot F}^{\alpha,q}_{p,{\rm D}}$.
\begin{pro}\label{thm 1.5}
Let $|\alpha|<1$, $1< p<\infty$ and $1\leq q\leq\infty$.  Then 
    ${\dot F}^{\alpha,q}_{p,{\rm D}}=\overline{\dot{\mathcal F}^{\alpha,q}_{p,{\rm D}}},$
    where $\overline{\dot{\mathcal F}^{\alpha,q}_{p,{\rm D}}}$ is the collection of distributions $f\in (\dot{\mathcal F}^{-\alpha,q'}_{p',{\psi}})^\prime$ such that there exists a sequence $\{f_n\}_{n=1}^\infty$ in $L^2(\mathbb R^N,\omega)$ with $\|f_n-f_m\|_{\dot{\mathcal F}^{\alpha,q}_{p,{\rm D}}}\rightarrow 0$ as $n,m\rightarrow \infty.$ Moreover, $f_n$ converges to $f$ in $(\dot{\mathcal F}^{-\alpha,q'}_{p',{\psi}})^\prime.$ 
\end{pro}

We will establish the connection between the Lipschitz spaces $\Lambda^\beta$ and the commutator of the Riesz transform $R_{{\rm D},\ell}:=T_\ell (-\triangle_{\rm D})^{-{1\over2}}$ associated to the Dunkl Laplacian $\triangle_{\rm D}$. Comparing to the classical known results, the key step here is a finer decomposition of the underlying space via Dunkl metric to bypass the use of Fourier transform.
\begin{thm}\label{thm commutator Riesz}
Let $\ell\in\{1,\ldots,  N \}$.  Let $1<p<\infty$ and $0<\beta<1$. 
The function $b\in \Lambda^\beta$ if and only if the commutator $[b,R_{{\rm D},\ell}]$ is a bounded operator from $L^p(\mathbb {R}^ N, \omega )$ to $\dot F_{p,{\rm D}}^{\beta,\infty}$. Moreover, 
$$\|b\|_{\Lambda^\beta}\sim \|[b,R_{{\rm D},\ell}]\|_{L^p_\omega \to  \dot F_{p,{\rm D}}^{\beta,\infty}}.$$
\end{thm}

Recall the fractional operator $\triangle_{\rm D}^{-\alpha/2}$ with $0<\alpha<{\bf N}$, which is defined by
$$ \triangle_{\rm D}^{-\alpha/2}f(x) = {1\over \Gamma(\alpha/2)} \int_0^\infty {e^{t\triangle_{\rm D}}}(f)(x) {dt\over t^{1-\alpha/2}}. $$
    Similarly we are also able to obtain the following estimates of the Lipschitz space $\Lambda^\beta$ via the commutator $[b,\triangle_{\rm D}^{-\alpha/2}](f)(x):= b(x)\triangle_{\rm D}^{-\alpha/2}(f)(x) - \triangle_{\rm D}^{-\alpha/2}(bf)(x) $.

\begin{thm}\label{main 2}
Let $1<p<q<\infty$, $0<\beta<1$ and  $\frac1p-\frac1q=\frac{\alpha}{\bf N}$. 
If $b\in \Lambda^\beta$, then the commutator $[b,\triangle_{\rm D}^{-\alpha/2}]$ is a bounded operator from 
$L^p(\mathbb {R}^N,\omega)$ to $\dot F_{q,{\rm D}}^{\beta,\infty}$. Moreover, $\|[b,\triangle_{\rm D}^{-\alpha/2}]\|_{L^p_\omega \to  \dot F_{q,{\rm D}}^{\beta,\infty}}\lesssim \|b\|_{\Lambda^\beta}$.

\end{thm}
Finally, when the parameter $\beta=0$ we can demonstrate that the  BMO space associated with the Euclidean metric  and Dunkl measure $\omega$ is the suitable replacement for $\Lambda^0$. To be more precise, we recall that ${\rm BMO}_{\rm Dunkl}(\mathbb{R}^N,\omega) $ is the set of all 
$b\in L^1_{loc}(\mathbb R^N,\omega)$
such that 
$$ \|b\|_{{\rm BMO}_{{\rm Dunkl}} } :=\sup_{B} {1\over \omega(B) }\int_{B}|b(x)-b_B| d\omega(x)<\infty,
$$
where the sup is taken over all balls $B$ in $\mathbb{R}^N$. 
\begin{thm}\label{main 1}
Suppose $b\in L^1_{loc}(\mathbb R^N,\omega)$, $0<\alpha<{\bf N}$, $1<p<{{\bf N}\over \alpha}$ and ${1\over q} = {1\over p} - {\alpha\over {\bf N}}$.
If $b\in {\rm BMO}_{{\rm Dunkl}} (\mathbb R^N,\omega)$, then $[b,\triangle_{\rm D}^{-\alpha/2}]$ is a bounded map from 
$L^p(\mathbb{R}^N,\omega)$ to $L^q(\mathbb{R}^N,\omega)$ with  $\|[b,\triangle_{\rm D}^{-\alpha/2}]\|_{L^p_\omega \to L^q_\omega}\lesssim \|b\|_{{\rm BMO}_{{\rm Dunkl}} }$.
\end{thm}

Here we point out that it is still an open question whether the reverse direction of Theorems \ref {main 2} or \ref{main 1} is true.

Throughout the paper we use the notation $X\lesssim Y$ to denote that there is an absolute constant $C$ so that $X\leq CY$.  If we write $X\sim Y$, then we mean that $X\lesssim Y$ and $Y\lesssim X$. 

%%%%%%%%%%%%%%%%%%%%%%%%%%%%%%%
%%%%%%%%%%%%%%%%%%%%%%%%%%%
%%%%%%%%%%%%%%%%%%%%%%%%%%%%%%%

\section{The characterization of the Lipschitz spaces}

We first show the characterization of the Lipschitz spaces.

\begin{proof}[The proof of Theorem \ref{lip and lipd}]
{Let $X=\mathbb R^N/ G$ be the space of orbits equipped with the metric $d(\mathcal{O}(x),\mathcal{O}(y))=d(x,y)$
and the measure ${\bf m}(A)=\omega\big(\cup_{\mathcal{O}(x)\in A} \mathcal{O}(x)\big)$. So $(X, d, {\bf m})$ is the space
of homogeneous type of in the sense of Coifman--Weiss (see \cite[page 2403]{ADH}). }
Hence the proofs of  the characterization of the Lipschitz spaces $\Lambda^\beta$ and  $\Lambda^\beta_d$ are similar.
It suffices to prove \eqref{e4.0},  \eqref{e4.1} and \eqref{e4.2} are equivalent.
We first show that \eqref{e4.0} implies \eqref{e4.2}.
For every $B\subseteq \mathbb R^ N$, we have that 
\begin{align*}
&\frac 1{\ell(B)^{\beta}}\bigg(\frac1{\omega(B)}\int_B |f(x)-f_B|^q d\omega(x)\bigg)^{\frac1q} \\
&\leq \frac 1{\ell(B)^{\beta}}\Bigg(\frac1{\omega(B)}\int_B\bigg(\frac1{\omega(B)}\int_B |f(x)-f(y)| d\omega(y)\bigg)^q d\omega(x)\Bigg)^{\frac1q} \\
&\leq \Bigg(\frac1{\omega(B)}\int_B\bigg(\frac1{\omega(B)}\int_B \sup_{x\ne y} \frac{|f(x)-f(y)|}{\|x-y\|^\beta} d\omega(y)\bigg)^q d\omega(x)\Bigg)^{\frac1q} \\
&\leq \sup_{x\ne y} \frac{|f(x)-f(y)|}{\|x-y\|^\beta}.
\end{align*}
Hence, \eqref{e4.0} implies \eqref{e4.2}.

Next, by H\"older's inequality, we have 
$$\frac 1{\ell(B)^{\beta}}\frac1{\omega(B)}\int_B |f(x)-f_B|d\omega(x) 
\le \frac 1{\ell(B)^{\beta}}\bigg(\frac1{\omega(B)}\int_B |f(x)-f_B|^q d\omega(x)\bigg)^{\frac1q}.$$
Hence, we get that  \eqref{e4.2} implies \eqref{e4.1}.

To show  \eqref{e4.1} implies  \eqref{e4.0}, we need the following result.
Suppose $S_k$ are the approximation operator to the identity in the space of homogeneous type $(\mathbb R^N, \|\cdot\|, d\omega)$. Set $E_k=S_k-S_{k-1}$. The kernel of $E_k$ satisfies 
$$E_kf(x)=\int_{\Bbb R^N}E_k(x,y)f(y)d\omega(y)$$ with
 $E_k(x,y)$ for all $x,y\in \mathbb{R}^{N}$ satisfying the
following conditions, see \cite{HHLLT}:
\begin{itemize}
 \item[(i)] $\displaystyle|E_k(x,y)|\leqslant C\frac1{V_k(x)+V_k(y)+V(x,y)}\frac{{\mathfrak R}^{-k}}{{\mathfrak R}^{-k}+\|x-y\|};$
        \item[(ii)] $\displaystyle|E_k(x,y)-E_k(x',y)|\leqslant C\frac{\|x-x'\|)}{{\mathfrak R}^{-k}+\|x-x'\|}\frac1{V_k(x)+V_k(y)+V(x,y)}\frac{{\mathfrak R}^{-k}}{{\mathfrak R}^{-k}+\|x-y\|}$,
        \item[]for $\|x-x'\|\leqslant ({\mathfrak R}^{-k}+\|x-y\|)/2,$ similarly for $|E_k(x,y)-E_k(x,y')|$;
        \item[(iii)] $\displaystyle \int_{\Bbb R^N} E_k(x,y)d\omega(x)=0\qquad \text{for all}\ y\in \Bbb R^N;$
        \item[(iv)] $\displaystyle \int_{\Bbb R^N} E_k(x,y)d\omega(y)=0\qquad \text{for all}\ x\in \Bbb R^N.$
\end{itemize}

By the method in \cite[Lemma 4.3]{DeH} or \cite[Theorem 8.1]{Lin}, we have the fact  that
$\|f\|_{\Lambda^\beta}$ via \eqref{e4.0} is equivalent to the Besov norm $D\binfty$
in the space of homogeneous type $(\mathbb R^N, \|\cdot\|, d\omega)$:
$$\|f\|_{D\binfty}:=\sup\limits_{\substack{k\in\mathbb Z, \ x \in \mathbb R^N}}{\mathfrak R}^{-k\beta}|E_kf(x)|<\infty.$$ 
 Now for any $k\in \mathbb Z, \ x \in \mathbb R^N$, we have
 \begin{align*}
 &{\mathfrak R}^{-k\beta}|E_kf(x)| \\
 &=  {\mathfrak R}^{-k\beta}\bigg|\int_{\Bbb R^N}E_k(x,y)(f(y) -f_{B(x,{\mathfrak R}^k)})d\omega(y)\bigg|\\
 &\leq C {\mathfrak R}^{-k\beta} \int_{\Bbb R^N} {1\over \omega(B(x,\|x-y\|))}\frac{{\mathfrak R}^k}{{\mathfrak R}^k+\|x-y\|}|f(y) -f_{B(x,{\mathfrak R}^k)}|d\omega(y)\\
 &\leq C {\mathfrak R}^{-k\beta} \sum_{j=1}^\infty \int_{B(x,2^j{\mathfrak R}^k)\backslash  B(x,2^{j-1}{\mathfrak R}^k) } {1\over \omega(B(x, \|x-y\|))}\frac{{\mathfrak R}^k}{{\mathfrak R}^k+\|x-y\|}|f(y) -f_{B(x,{\mathfrak R}^k)}|d\omega(y)\\
 &\quad+C {\mathfrak R}^{-k\beta}  \int_{B(x,{\mathfrak R}^k) } {1\over \omega(B(x, \|x-y\|))}\frac{{\mathfrak R}^k}{{\mathfrak R}^k+\|x-y\|}|f(y) -f_{B(x,{\mathfrak R}^k)}|d\omega(y)\\
 &\leq C  {\mathfrak R}^{-k\beta} \sum_{j=1}^\infty \int_{B(x,2^j{\mathfrak R}^k) } {1\over  \omega(B(x, 2^{j-1}{\mathfrak R}^k))}\frac{{\mathfrak R}^k}{{\mathfrak R}^k+2^{j-1}{\mathfrak R}^k}|f(y) -f_{B(x,{\mathfrak R}^k)}|d\omega(y)\\
 &\quad+C {\mathfrak R}^{-k\beta} {1\over \omega((B(x,{\mathfrak R}^k)))}  \int_{B(x,{\mathfrak R}^k) } |f(y) -f_{B(x,{\mathfrak R}^k)}|d\omega(y)\\
 &\leq C\sup_{B\subseteq \mathbb R^ N }\frac 1{\ell(B)^{\beta}\omega(B)}\int_{B} |f(x)-f_B| d\omega(x).
 \end{align*}
Hence, \eqref{e4.0},  \eqref{e4.1} and \eqref{e4.2} are equivalent.
The proof is completed.
\end{proof}

We next show that the Lipschitz space associated to the metric $d$ is a proper subset of  the Lipschitz space associated the Euclidean norm.

\begin{proof}[Proof of Theorem \ref{lip comp}]
Suppose $f\in \Lambda^\beta_d$. Then $d(x,y)\le \|x-y\|$ gives
$$\sup_{d(x,y)\ne 0}\frac{|f(x)-f(y)|}{\|x-y\|^\beta} \le \sup_{d(x,y)\ne 0}\frac{|f(x)-f(y)|}{d(x,y)^\beta}<\infty.$$
Note that $d(x,y)=0$ implies that $y=\sigma(x)$ for some $\sigma\in G$.
It is clear that
$$\sup_{x\ne y}\frac{|f(x)-f(y)|}{\|x-y\|^\beta}\le \sup_{d(x,y)\ne 0}\frac{|f(x)-f(y)|}{\|x-y\|^\beta}
            + \sup_{\{y=\sigma(x) : \sigma\in G\} \atop y\ne x}\frac{|f(x)-f(y)|}{\|x-y\|^\beta}<\infty$$
since $G$ is finite group. Hence $f\in \Lambda^\beta$. Thus, we first obtain 
$$ \Lambda^\beta_d\subset \Lambda^\beta. $$

Next, we consider an example in the real line $\mathbb R$ to show that the containment is proper. 
Take the function $f\in C^1(\mathbb R)$ given by
\[
    f(x)=\left\{
                \begin{array}{ll}
                  -1,\quad x\leq -{\pi\over2};\\[5pt]
                   \sin x, \ \ \ -{\pi\over2}<x<{\pi\over2};\\[5pt]
                  1,\quad x\geq{\pi\over2}.\\
                \end{array}
              \right.
  \]
Then it is clear that this $f$ belongs to $\Lambda^\beta$ for every $\beta\in(0,1)$.   However, we take $x>\pi$ and we take $y=-x-\delta$ with $\delta$ a small positive number, for example, $\delta<100^{-1}$. Then we have 
$$d(x,y) =\delta.$$
It is clear that 
$$|f(x)-f(y)|= 2> \delta^\beta = d(x,y)^\beta$$
for every $\beta\in(0,1)$.
Hence, 
$$\sup_{d(x,y)\ne 0}\frac{|f(x)-f(y)|}{d(x,y)^\beta}=\infty,$$
which shows that
this $f$ does not belong to $\Lambda^\beta_d$ and so in general we have $\Lambda^\beta_d\subsetneq \Lambda^\beta$.
The proof of Theorem \ref{lip comp} is complete.
\end{proof}

%%%%%%%%%%%%%%%%%%%%%%%%%%%%%%%
%%%%%%%%%%%%%%%%%%%%%%%%%%%
%%%%%%%%%%%%%%%%%%%%%%%%%%%%%%%

\section{Triebel--Lizorkin spaces and the Calder\'on reproducing formula in $\dot{\mathcal F}^{\alpha,q}_{p,{\rm D}}$}

We begin by collecting some notation in the Dunkl setting. Let $B(x,r):=\{y\in \mathbb{R}^{N}:\|x-y\|<r\}$ denote the ball with
center  $ x\in\mathbb{R}^{N}$  and radius $r>0$. 
Clearly,
$$\omega(B(tx, tr)) = t^\textbf{N}\omega(B(x, r))$$
and the following change of variables formula holds
$\int_{\mathbb{R}^{N}}f(x)d\omega(x)=\int_{\mathbb{R}^{N}}\frac{1}{t^\textbf{N}}f(\frac{x}{t})d\omega(x)$
for $f \in L^1(\mathbb{R}^N,\omega)$, $t>0$.

Observe that for $x\in \mathbb{R}^N$ and $r>0,$
$$\omega(B(x,r))\sim r^N \prod_{\alpha\in R}\big(|\langle\alpha,x\rangle|+r\big)^{\kappa(\alpha)}.$$
%Set
%$$V(x,y,r):=\max\{\omega(B(x,r)),\omega(B(y,r))\}.$$

\subsection{The Calder\'on Reproducing Formula in $L^2_\omega$}

The Poisson semigroup is given by
$$P_tf(x)=\pi^{-\frac{1}{2}}\int^\infty_0 e^{-u}\exp\Big(\frac{t^2}{4u}\triangle_D\Big)f(x)\frac{du}{u^{\frac{1}{2}}}$$
and $u(x, t)=P_tf(x),$ so-called the Dunkl Poisson integral, solves
the boundary value problem 
$$\begin{cases}
       & (\partial_t^2 + \triangle_D)u(x,t)=0,\\[3pt]
       & u(x,0)=f(x)
    \end{cases}
$$
in the half-space $\mathbb R^N_+,$ see \cite{ADH}. Let $p_t$ be the Poisson kernel and let $q_t:=t\partial_{t}p_t$.
We remark that 
   $$    |\partial^m_t\partial^\alpha_x\partial^\beta_y p_t(x,y)|
    \lesssim t^{-m-|\alpha|-|\beta|}\frac1{V(x,y,t+d(x,y))}\frac{t}{t+\|x-y\|}.
  $$
    These estimates indicate that $q_{t}(x,y)$ for all $x,y\in \mathbb{R}^{N}$ and $t>0,$ satisfy the following:
   \begin{eqnarray*}
    &\textup{(i)}& |q_t(x,y)|\leq  {1\over V(x,y,t+d(x,y))}\frac{t}{t+\|x-y\|};\\
    &\textup{(ii)}& |q_t(x,y)-q_t(x',y)|\\
        &&\leq {\|x-x'\|\over t} \Big({1\over V(x,y,t+d(x,y))}\frac{t}{t+\|x-y\|}+{1\over V(x',y,t+d(x',y))}\frac{t}{t+\|x'-y\|}\Big);\\
    &\textup{(iii)}& |q_t(x,y)-q_t(x,y')|\\
        &&\leq {\|y-y'\|\over t}\Big({1\over V(x,y,t+d(x,y))}\frac{t}{t+\|x-y\|}+{1\over V(x,y',t+d(x,y'))}\frac{t}{t+\|x-y'\|}\Big);\\
    &\textup{(iv)}& \int_{\mathbb{R}^{N}}q_t(x,y) d\omega(y)
        = \int_{\mathbb{R}^{N}}q_t(x,y) d\omega(x)=0.
    \end{eqnarray*}
   
We begin with the following Calder\'on reproducing formula provided in \cite{ADH}.

\begin{thm}\label{CRF}
For $f\in L^2(\mathbb R^N, \omega)$ we have that
\begin{eqnarray}
\label{CRF1}f(x)=\int_0^\infty \psi_{t}\ast q_{t}\ast f(x)\frac{dt}{t},
\end{eqnarray}
where $\psi(x)$ is a radial Schwartz function supported in the unit ball $B(0,1)$, $\psi_t(x,y)=\tau_x\psi_t(-y)$, $\psi_t(x)=t^{-\bf N}\psi(x/t)$ and $r_t\ast f(x)= f(x)=\int_{\mathbb{R}^{N}}r_{t}(x,y)f(y)d\omega(y)$, $r_t\in\{q_t,\psi_t\}$. Moreover, $\psi_{t}(x,y)$ for all $x,y\in \mathbb{R}^{N}$, $t>0$ similar conditions as above hold, but with the support $\{(x,y): d(x,y)\leq t\}.$
\end{thm}

Applying Coifman's decomposition of the identity on $L^2(\mathbb R^N, \omega)$ 
gives
$$I=T_M + R_1 + R_M.$$ That is, 

$$f(x)=\int_0^\infty \psi_t\ast q_t\ast f(x)\frac{dt}{t}=T_M + R_1 + R_M,$$
where
$$T_M(f)(x)=-\ln {\mathfrak R}\sum\limits_{j=-\infty}^\infty\sum\limits_{Q\in Q^j}w(Q) \psi_{j}(x,x_{Q})q_{j}
f(x_{Q}),$$
$$R_1f(x)=-\sum\limits_{j=-\infty}^\infty\int_{{\mathfrak R}^{-j}}^{{\mathfrak R}^{-j+1}}
\Big[\psi_t\ast q_tf(x)- \psi_{j}\ast q_{j}f(x)\Big]\frac{dt}{t}$$
and
$$R_M(f)(x)=-\ln {\mathfrak R}\sum\limits_{j=-\infty}^\infty\sum\limits_{Q\in Q^j}\int_{Q}
\Big[\psi_{j}(x,y)q_{j}f(y)-\psi_{j}(x,x_{Q})q_{j}
f(x_{Q})\Big]d\omega(y),$$ where $q_{j}f(x)$ is always denoted by
$q_{j}\ast f(x), \psi(x,y)$ is the translation of $\psi$ by $x$ and
$\psi_j=\psi_{{\mathfrak R}^j}, q_{j}=q_{{\mathfrak R}^j},$ with $$1<{\mathfrak R}\leq {\mathfrak R}_0$$ for some
fixed ${\mathfrak R}_0$, $Q^j$ is the collection of all $``{\mathfrak R}-{\rm dyadic\ cubes}"\ Q$ with side length ${\mathfrak R}^{-M-j}$ with $M$ is some fixed large integer, and $x_{Q}$ is any fixed point in the cube $Q.$

 We also have the following discrete weak-type Calder\'on reproducing formula in \cite{HHLLT}.

\begin{thm}
    If $f\in L^2(\mathbb R^N, \omega)\cap L^p(\mathbb R^N, \omega), 1<p<\infty,$ then there exists a function $h\in L^2(\mathbb R^N, \omega)\cap L^p(\mathbb R^N, \omega),$ such that $\|f\|_{L^2_\omega}\sim \|h\|_{L^2_\omega}$ and $\|f\|_{L^p_\omega}\sim \|h\|_{L^p_\omega},$
    \begin{equation}\label{CR-1}
    \begin{aligned}
    f(x)=\sum\limits_{j=-\infty}^\infty\sum\limits_{Q\in Q^j}\omega(Q)
    \psi_{j}(x,x_{Q})q_{j}
    h(x_{Q}),
    \end{aligned}
    \end{equation}
    where the series converges in $L^2(\mathbb R^N, \omega)\cap L^p(\mathbb R^N, \omega)$ with $\psi_j=\psi_{{\mathfrak R}^j}, q_{j}=q_{{\mathfrak R}^j}, 1<{\mathfrak R}\leq {\mathfrak R}_0$ for some fixed ${\mathfrak R}_0,$
    $Q^j$ is the collection of all ``${\mathfrak R}$-dyadic cubes" $Q$ with side length ${\mathfrak R}^{-M-j}$ with $M$ is some fixed large integer, and $x_{Q}$ is any fixed point in the cube $Q.$
\end{thm}

Considering $(\mathbb R^N, \|\cdot\|, \omega)$ as a space of homogeneous type in the sense of Coifman and Weiss,  we use the well known discrete Calder\'on reproducing formula in $L^2(\Bbb R^N, \omega)$ to obtain the reproducing formula. 

\begin{thm}\label{thm 3.3}
Let $q_{{\mathfrak R}^j}$ and $\psi_{{\mathfrak R}^j}$ be given as in Theorem \ref{CRF}. For $|\alpha|<1$ and $1\le p, q< \infty$, if $f\in \dot{\mathcal F}^{\alpha,q}_{p,{\rm D}}$,
then there exists a function $h\in \dot{\mathcal F}^{\alpha,q}_{p,{\rm D}},$ such that $\|f\|_{L^2_\omega}\sim \|h\|_{L^2_\omega}$ and $\|f\|_{\dot{\mathcal F}^{\alpha,q}_{p,{\rm D}}}\sim \|h\|_{\dot{\mathcal F}^{\alpha,q}_{p,{\rm D}}},$
  \begin{align*}
    f(x)=\sum\limits_{j=-\infty}^\infty\sum\limits_{Q\in Q^j}\omega(Q)
    \psi_{{\mathfrak R}^j}(x,x_{Q})q_{{\mathfrak R}^j}
    h(x_{Q}),
    \end{align*}  
 where the series converges in $\dot{\mathcal F}^{\alpha,q}_{p,{\rm D}}$.
\end{thm}

Theorem \ref{thm 3.3} implies the following duality
estimate which will be a key idea for developing the Dunkl--Triebel--Lizorkin
space theory:
\begin{thm}\label{thm 3.4}
Let $|\alpha|<1$, $1< p<\infty$ and $1\leq q\leq\infty$.
    For $f\in \dot{\mathcal F}^{\alpha,q}_{p,{\rm D}}$ and $g\in \dot{\mathcal F}^{-\alpha,q'}_{p',{\psi}}$, then there exists a constant $C$ such that
    $$|\langle f, g\rangle|\leqslant C \|f\|_{\dot{\mathcal F}^{\alpha,q}_{p,{\rm D}}}\|g\|_{\dot{\mathcal F}^{-\alpha,q'}_{p',{\psi}}},$$
   where $p'$ and $q'$ are conjugate numbers of $p$ and $q$, respectively.
\end{thm}

Theorem \ref{thm 3.4} means that each function $f\in
L^2(\mathbb R^N,\omega)$ with $\|f\|_{\dot{\mathcal F}^{\alpha,q}_{p,{\rm D}}}<\infty$ can be considered as a continuous linear functional on $\dot{\mathcal F}^{-\alpha,q'}_{p',{\psi}}$. Therefore, one can consider
$\dot{\mathcal F}^{-\alpha,q'}_{p',{\psi}}$ as a new test
function space and define the Dunkl--Triebel--Lizorkin space  ${\dot F}^{\alpha,q}_{p,{\rm D}}$ as the
collection of appropriate distributions on $\dot{\mathcal F}^{-\alpha,q'}_{p',{\psi}}$. 
 
\subsection{ Calder\'on Reproducing Formula in $\dot{\mathcal F}^{\alpha,q}_{p,{\rm D}}$}

We begin with
the following definition of the test functions in the space of
homogeneous type $(\Bbb R^N,\|\cdot\|,\omega)$:

\begin{defn}\label{test} A function $f(x)$ defined on $\Bbb R^N$ is a test function if there exits a constant $C$ such that for $0<\beta\leqslant 1, \gamma>0, r>0$ and $x_0\in \mathbb R^N,$
    \begin{enumerate}
        \item[(i)] $\displaystyle f(x)\leqslant \frac C{V(x, r+\|x-x_0\|)}\Big(\frac{r}{r+\|x-x_0\|}\Big)^\gamma;$\\[4pt]
        \item[(ii)] $\displaystyle |f(x)-f(x')|\leqslant C\Big(\frac{\|x-x'\|}{r+\|x-x_0\|}\Big)^\beta\frac{1}{V(x,r+\|x-x_0\|)}\Big(\frac{r}{r+\|x-x_0\|}\Big)^\gamma, \\ \quad \text{for}\quad \|x-x'\|\leqslant \frac{1}{2}(r+\|x-x_0\|);$
        \item[(iii)] $\displaystyle \int_{\mathbb R^N} f(x)d\omega(x)=0.$
    \end{enumerate}
    We denote such a test function by $f\in \mathcal M(\beta,\gamma,r,x_0)$ and $\|f\|_{\mathcal M(\beta,\gamma,r,x_0)},$ the norm in $\mathcal M(\beta,\gamma,r,x_0),$ is the smallest $C$ satisfying the above conditions \textnormal{(i)} and \textnormal{(ii)}.
\end{defn}

 Let $\{S_k\}_{k\in \Bbb Z}$ be Coifman's approximation to the identity.
 Then the kernels $S_k(x,y)$ of $S_k$ satisfy the following properties:
\begin{enumerate}
    \item[(i)] $S_k(x,y)=S_k(y,x);$
    \item[(ii)] $S_k(x,y)=0$ if $\|x-y\|>{\mathfrak R}^{4-k}$\ \ and\ \ $\displaystyle |S_k(x,y)|\leqslant \frac C{V_k(x)+V_k(y)};$
    \item[(iii)] $\displaystyle |S_k(x,y)-S_k(x',y)|\leqslant C\frac{{\mathfrak R}^k\|x-x'\|}{V_k(x)+V_k(y)}$\qquad for $\|x-x'\|\leqslant {\mathfrak R}^{8-k};$\\[3pt]
    \item[(iv)] $\displaystyle |S_k(x,y)-S_k(x,y')|\leqslant C\frac{{\mathfrak R}^k\|y-y'\|}{V_k(x)+V_k(y)}$\qquad for $\|y-y'\|\leqslant {\mathfrak R}^{8-k};$
    \item[(v)] $\displaystyle \big|[S_k(x,y)-S_k(x',y)]-[S_k(x,y')-S_k(x',y')]\big|\leqslant C\frac{{\mathfrak R}^k\|x-x'\|r^k\|y-y'\|}{V_k(x)+V_k(y)}$
    \item[] for $\|x-x'\|\leqslant {\mathfrak R}^{8-k}$ and $\|y-y'\|\leqslant {\mathfrak R}^{8-k};$
    \item[(vi)] $\displaystyle \int_{\Bbb R^N} S_k(x,y)d\omega(x)=1\qquad \text{for all}\ y\in \Bbb R^N;$
    \item[(vii)] $\displaystyle \int_{\Bbb R^N} S_k(x,y)d\omega(y)=1\qquad \text{for all}\ x\in \Bbb R^N.$
\end{enumerate}
 Here and
throughout, $V_k(x)$ denotes the measure
$\omega(B(x,{\mathfrak R}^{-k}))$ for ${\mathfrak R}>1, k\in \Bbb Z$ and $x\in \Bbb R^N$. We
also denote by $V(x,y)=\omega(B(x,\|x-y\|))$ for $x,y\in \Bbb R^N$. 

Applying Coifman's decomposition of the identity and Calder\'on-Zygmund theory, the discrete Calder\'on reproducing formula in space of homogeneous type is the following result.

\begin{thm}\label{dCRh}
    Let $\{S_k\}_{k\in \Bbb Z}$ be Coifman's approximation to the identity and set
    ${D}_k := {S}_k - {S}_{k-1}$. Then there exists a family of operators
    $\{\widetilde{D}_k\}_{k\in \Bbb Z}$ such that for any fixed $x_{Q}\in Q$ with $k\in \Bbb Z$ and $Q$ are ``${\mathfrak R}$-dyadic cubes" with the side length ${\mathfrak R}^{-M-k},$
    $$f(x)=\sum\limits_{k=-\infty}^\infty \sum\limits_{Q\in Q^k}\omega(Q){\widetilde D}_k(x,x_{Q})D_k(f)(x_{Q}),$$
    where the series converge in $L^p(\mathbb{R}^N, \omega)$, $1<p<\infty$, $\mathcal M(\beta,\gamma,r,x_0),$ and in $(\mathcal M(\beta,\gamma,r,x_0))^\prime,$ the dual of in $\mathcal M(\beta,\gamma,r,x_0)$.  Moreover, the kernels of the operators $\widetilde{D}_k$
    satisfy the following conditions:
    \begin{enumerate}
        \item[(i)] $\displaystyle|\widetilde{D}_k(x,y)|\leqslant C\frac1{V_k(x)+V_k(y)+V(x,y)}\frac{{\mathfrak R}^{-k}}{{\mathfrak R}^{-k}+\|x-y\|};$
        \item[(ii)] $\displaystyle|\widetilde{D}_k(x,y)-\widetilde{D}_k(x',y)|\leqslant C\frac{\|x-x'\|}{{\mathfrak R}^{-k}+\|x-x'\|}\frac1{V_k(x)+V_k(y)+V(x,y)}\frac{{\mathfrak R}^{-k}}{{\mathfrak R}^{-k}+\|x-y\|}$,
        \item[]for $\|x-x'\|\leqslant ({\mathfrak R}^{-k}+\|x-y\|)/2;$
        \item[(iii)] $\displaystyle \int_{\Bbb R^N} \widetilde{D}_k(x,y)d\omega(x)=0\qquad \text{for all}\ y\in \Bbb R^N;$
        \item[(iv)] $\displaystyle \int_{\Bbb R^N} \widetilde{D}_k(x,y)d\omega(y)=0\qquad \text{for all}\ x\in \Bbb R^N.$
    \end{enumerate}
    Similarly, there exists a family of linear operators
    $\{\widetilde{\widetilde{D}}_k\}_{k\in \Bbb Z}$ such that for any fixed $x_{Q}\in Q,$
    $$f(x)=\sum\limits_{k=-\infty}^\infty \sum\limits_{Q\in Q^k}\omega(Q)D_k(x,x_{Q})\widetilde{\widetilde{D}}_k(f)(x_{Q}),$$
    where the kernels of the operators $\widetilde{\widetilde{D}}_k$
    satisfy the above conditions {\rm (i), (ii), (iii),} and {\rm (iv)} with $x$ and $y$ interchanged.
 \end{thm}

It is well known that in the classical case, almost orthogonality estimates are fundamental tools in discrete Calder\'on reproducing formula.
The following result provides such a tool in the Dunkl setting, see \cite[Lemmas 2.10 and 2.11]{HHLLT}. 
\begin{lem}\label{T aoe}
    Let $T$ be a Dunkl--Calder\'on--Zygmund operator satisfying $T(1)=T^*(1)=0$. Then
    \begin{align*}
    &\bigg|\int_{\Bbb R^N}\int_{\Bbb R^N} D_k(x,u)K(u,v)D_j(v,y)
    d\omega(u)d\omega(v)\bigg|\\
    &\quad \lesssim {\mathfrak R}^{-|k-j|\varepsilon'}
    \frac1{V(x,y, {\mathfrak R}^{(-j)\vee (-k)}+d(x,y))}\Big(\frac{{\mathfrak R}^{(-j)\vee (-k)}}{{\mathfrak R}^{(-j)\vee (-k)}+d(x,y)}\Big)^\gamma,
    \end{align*}
    where $\gamma, \varepsilon'\in (0,\varepsilon)$.
\end{lem}

We next need the definition of smooth molecules.  %Here $\mathbb M(\beta, \gamma, r, x_0)$ and  $\widetilde{\mathbb M}(\beta, \gamma, r, x_0)$ are defined by following
\begin{defn}\label{sm}
    A function $f(x)$ is a smooth molecule for $0<\beta\leqslant 1, \gamma>0, r>0$ and some fixed $x_0\in \mathbb{R}^N,$ if $f(x)$ satisfies the following conditions:
        \begin{equation}\label{sm1.17}
        |f(x)|\le C \frac{1}{V(x,x_0,r+d(x,x_0))}\Big(\frac{r}{r+{\|x-x_0\|}}\Big)^\gamma;
        \end{equation}
        \begin{equation}\label{sm1.18}
        \begin{aligned}
        |f(x)-f(x')|&\le C  \Big(\frac{\|x-x'\|}{r}\Big)^\beta
        \Big\{   \frac{1}{V(x,x_0,r+d(x,x_0))}\Big(\frac{r}{r+{\|x-x_0\|}}\Big)^\gamma\\
        &\qquad\qquad + \frac{1}{V(x',x_0,r+d(x',x_0))}\Big(\frac{r}{r+{\|x'-x_0\|}}\Big)^\gamma\Big\};
        \end{aligned}
        \end{equation}
        \begin{equation}\label{sm1.19}
        \int_{\mathbb{R}^N} f(x) d\omega(x)=0.
        \end{equation}
    If $f(x)$ is a smooth molecule, we denote $f(x)$ by $f\in \mathbb M(\beta, \gamma, r, x_0)$ and define the norm of $f$ by
    $$\|f\|_{\mathbb M(\beta, \gamma, r, x_0)}:=\inf\{C: (\ref{sm1.17})-(\ref{sm1.18})\ {\rm hold}\}.$$
\end{defn}
Observe that $t\partial_t p_t(x,y)$ with $p_t,$ the Poisson kernel,
is a smooth molecule with $\beta, \gamma=1, r=t, x_0=y$ for fixed
$y,$ and it is also a smooth molecule for $x_0=x$ for $x$ is
fixed.

\begin{defn}\label{wsm}
    A function $f(x)$ is said to be a \textit{Dunkl smooth molecule} for $0<\beta\leqslant 1, \gamma>0, r>0$ and some fixed $x_0\in \mathbb{R}^N,$ if $f(x)$ satisfies the following conditions:
    \begin{equation}\label{wsm1.20}
    |f(x)|\leqslant C \frac{1}{V(x,x_0,r+d(x,x_0))}\Big(\frac{r}{r+{d(x,x_0)}}\Big)^\gamma;
    \end{equation}
    \begin{equation}\label{wsm1.21}
         \begin{aligned}
    |f(x)-f(x')|&\leqslant C  \Big(\frac{\|x-x'\|}{r}\Big)^\beta
    \Big\{   \frac{1}{V(x,x_0,r+d(x,x_0))}\Big(\frac{r}{r+{d(x,x_0)}}\Big)^\gamma\\
    &\qquad\qquad + \frac{1}{V(x',x_0,r+d(x',x_0))}\Big(\frac{r}{r+{d(x',x_0)}}\Big)^\gamma\Big\};
         \end{aligned}
    \end{equation}
    \begin{equation}\label{f1}
    \int_{\mathbb{R}^N} f(x) d\omega(x)=0.
    \end{equation}
    If $f(x)$ is a Dunkl smooth molecule, we denote $f(x)$ by $f\in \widetilde {\mathbb M}(\beta, \gamma, r, x_0)$ and define the norm of $f$ by
    $$\|f\|_{\widetilde{\mathbb M}(\beta, \gamma, r, x_0)}:=\inf\{C: (\ref{wsm1.20})- (\ref{wsm1.21})\ {\rm hold} \}.$$
\end{defn}
To clarify the difference between Definitions \ref{sm} and \ref{wsm} we remark that all $\|x-x_0\|$ and $\|x'-x_0\|$ in the classical smooth molecule definition are simply replaced by $d(x,x_0)$ and $d(x',x_0)$ in the Dunkl smooth molecule, respectively.

\begin{lem}\label{aoe}
    Let $x,y\in \mathbb{R}^N$ and $\varepsilon_0,t,s>0$ with $t\ge s.$  Suppose that $f_t(x,\cdot)$ is a Dunkl smooth molecule in $\widetilde{\mathbb M}(\varepsilon_0,\varepsilon_0, t, x)$ and
    $g_s(\cdot,y)$ is a  smooth molecule  in $\mathbb M(\varepsilon_0,\varepsilon_0, s, y)$. Then for any $0<\varepsilon_1,\varepsilon_2<\varepsilon_0$,
    there exists $C=C(\varepsilon_0,\varepsilon_1,\varepsilon_2)>0$, such that for all $t\ge s>0,$
    \begin{equation}\label{aoe1}
    \int_{\Bbb R^N}f_t(x,u)g_s(u,y)d\omega(u)\le C \bigg(\frac{s}{t}\bigg)^{\varepsilon_1} \frac1{V(x,y,t+d(x,y))}\Big(\frac{t}{t+d(x,y)}\Big)^{\varepsilon_2}.
    \end{equation}
    If $f_t(x,\cdot)$ and $g_s(\cdot,y)$ both are smooth molecules in $\mathbb M(\varepsilon_0,\varepsilon_0, t, x)$ and $\mathbb M(\varepsilon_0,\varepsilon_0, s, y),$ respectively,
    then for any $0<\varepsilon_1,\varepsilon_2<\varepsilon_0$, there exists $C=C(\varepsilon_0,\varepsilon_1,\varepsilon_2)>0$, such that for all $t, s>0,$
    \begin{align}\label{aoe2}
    &\int_{\Bbb R^N}f_t(x,u)g_s(u,y)d\omega(u)\\
    &\le C \bigg(\frac{s}{t}\wedge \frac{t}{s}\bigg)^{\varepsilon_1} \frac1{V(x,y,(t\vee s)+d(x,y))}\Big(\frac{t\vee s}{(t\vee s)+\|x-y\|}\Big)^{\varepsilon_2},\nonumber
    \end{align}
    where $a\wedge b=\min\{a,b\}$ and $a\vee b=\max\{a,b\}.$
\end{lem}

We now prove Theorem \ref{thm 3.3}.
\begin{proof}[Proof of Theorem \ref{thm 3.3}]
We first show that $T_M$ is bounded on $\dot{\mathcal F}^{\alpha,q}_{p,{\rm D}}$.
Recall that $$T_M(f)(x)=-\ln {\mathfrak R}\sum\limits_{k\in \mathbb Z}\sum\limits_{Q\in Q^k}w(Q) \psi_{k}(x,x_{Q})q_{k}
f(x_{Q}).$$ 
Let $f\in \dot{\mathcal F}^{\alpha,q}_{p,{\rm D}}$.
For $1\le q<\infty$, we write 
\begin{align*}
\|T_M(f)\|_{\dot{\mathcal F}^{\alpha,q}_{p,{\rm D}}}
&=\Big\|\Big\{\sum_{k'\in \mathbb Z}\sum_{Q'\in Q^{k'}} \big({\mathfrak R}^{k'\alpha}|q_{k'}(T_Mf)(x_{Q'})|\big)^q\chi_{Q'}\Big\}^{1/q}\Big\|_{L^p_\omega}\\
&\le C\Big\|\Big\{\sum_{k'\in \mathbb Z} \sum_{Q'\in Q^{k'}}\Big({\mathfrak R}^{k'\alpha}\Big|q_{k'}\Big(\sum_{k\in \Bbb Z}\sum\limits_{Q\in Q^k}w(Q) \psi_{k}(x_{Q'},x_{Q})q_{k}
f(x_{Q})\Big)\Big|\Big)^q\chi_{Q'}\Big\}^{1/q}\Big\|_{L^p_\omega}\\
&\le C\Big\|\Big\{\sum_{k'\in \mathbb Z}  \sum_{Q'\in Q^k}\Big({\mathfrak R}^{k'\alpha}\sum_{k\in \Bbb Z}\sum\limits_{Q\in Q^k}w(Q) |q_{k'}\psi_k(x_{Q'},x_Q)q_{k}
f(x_{Q})|\Big)^q\chi_{Q'}\Big\}^{1/q}\Big\|_{L^p_\omega}.
\end{align*}
By Lemma \ref{aoe}, we choose $|\alpha|<\varepsilon<1$ such that
$$|q_{k'}\psi_{k}(x_{Q'},x_{Q})|
\le C {\mathfrak R}^{-|k'-k|\varepsilon} \frac1{V(x_{Q'},x_Q,{\mathfrak R}^{(-k\vee -k')}+d(x_{Q'},x_Q))}\Big(\frac{{\mathfrak R}^{(-k\vee -k')}}{{\mathfrak R}^{(-k\vee -k')}+\|x_{Q'}-x_Q\|}\Big)^{\varepsilon}.$$
Observing that $\omega(B(x,{\mathfrak R}^{-k\vee -k'}+d(x,x_{Q^k})))\sim \omega(B(y,{\mathfrak R}^{-k\vee -k'}+d(y,x_{Q^k})))$ for $x,y\in Q^{k'},$ hence,
        \begin{align*}
        &|q_{k'}\psi_{k}(x_{Q'},x_{Q})|\chi_{Q'}(x) \\
        &\lesssim C{\mathfrak R}^{-|k-k'|\varepsilon} \frac1{V(x_Q,x_{Q'},{\mathfrak R}^{-k\vee -k'}+d(x_{Q^{'}},x_{Q}))}
        \bigg(\frac{{\mathfrak R}^{-k \vee -k'}}{{\mathfrak R}^{-k \vee -k'} +d(x_{Q^{'}},x_{Q})}\bigg)^{\varepsilon} \\
        &\lesssim C{\mathfrak R}^{-|k-k'|\varepsilon} \frac1{\omega(B(x_{Q},{\mathfrak R}^{-k\vee -k'}+d(x,x_{Q}))}
        \bigg(\frac{{\mathfrak R}^{-k \vee -k'}}
        {{\mathfrak R}^{-k \vee -k'} +d(x,x_{Q})}\bigg)^{\varepsilon}.
        \end{align*}
 Applying $d(x,y)=\min\limits_{\sigma\in G}\|\sigma(x)-y\|$ gives
        \begin{align*}
        &|q_{k'}\psi_{k}(x_{Q'},x_{Q})|\chi_{Q'}(x) \\
        &\lesssim \sum\limits_{\sigma\in G} {\mathfrak R}^{-|k-k'|\varepsilon} \frac1{\omega(B(x_Q,{\mathfrak R}^{-k\vee -k'}+\|\sigma(x)-x_{Q}\|))}
        \bigg(\frac{{\mathfrak R}^{-k \vee -k'}}{{\mathfrak R}^{-k \vee -k'} +\|\sigma(x)-x_{Q}\|}\bigg)^{\varepsilon}\\
        &\lesssim \sum\limits_{\sigma\in G} {\mathfrak R}^{-|k-k'|\varepsilon} \frac1{\omega(B(\sigma(x),{\mathfrak R}^{-k\vee -k'}+\|\sigma(x)-x_{Q}\|))}
        \bigg(\frac{{\mathfrak R}^{-k \vee -k'}}{{\mathfrak R}^{-k \vee -k'} +\|\sigma(x)-x_{Q}\|}\bigg)^{\varepsilon}.
        \end{align*}
        Let $\theta$ satisfy that $\frac{\bf N}{{\bf N}+\varepsilon-|\alpha|}<\theta<p.$
        Then
   \begin{align*}
        &\sum\limits_{Q\in Q^k}w(Q)\frac1{\omega(B(\sigma(x),{\mathfrak R}^{-k\vee -k'}+\|\sigma(x)-x_{Q}\|))}
        \bigg(\frac {{\mathfrak R}^{-k\vee -k'}}{{\mathfrak R}^{-k\vee -k'}+\|\sigma(x)-x_{Q}\|}\bigg)^{\varepsilon}|q_{k}
f(x_{Q})|\\
        &\leqslant \bigg\{\sum\limits_{Q\in Q^k}w(Q)^\theta \frac1{\omega(B(\sigma(x),{\mathfrak R}^{-k\vee -k'}+\|\sigma(x)-x_{Q}\|))
            )^\theta}
        \bigg(\frac {{\mathfrak R}^{-k\vee -k'}}{{\mathfrak R}^{-k\vee -k'}+\|\sigma(x)-x_{Q}\|}\bigg)^{\theta\varepsilon}|q_{k}f(x_{Q})|^\theta\bigg\}^{1\over\theta}.
        \end{align*}
    Denote by $c_Q$ the center point of $Q$. Let $A_0=\{Q\in Q^k : \|c_Q-\sigma(x)\|\leqslant {\mathfrak R}^{-k\vee -k'}\}$
        and $A_\ell=\{Q\in Q^k : {\mathfrak R}^{\ell-1+(-k\vee -k')}<\|c_Q-\sigma(x)\|\leqslant {\mathfrak R}^{\ell+(-k\vee -k')}\}$ for $\ell\in \mathbb N$.
        We use \eqref{rd1.1} to obtain that for $Q\in Q^k,$
        $$\omega(Q)\chi_{Q}(z)\sim \omega(B(z,
        {\mathfrak R}^{-k}))\chi_{Q}(z)\sim
        \omega(B(\sigma(z), {\mathfrak R}^{-k}))\chi_{Q}(z)\quad\mbox{for }\sigma\in G$$
        and
        $$\omega(B(x_{Q},{\mathfrak R}^{-k\vee -k'}))\lesssim
        {\mathfrak R}^{[k+(-k\vee -k')]{\bf N}}\omega(B(x_{Q},{\mathfrak R}^{-k})).$$
        Hence,
        \begin{align*}
        &\sum\limits_{Q\in Q^k} w(Q)^\theta \frac1{\omega(B(\sigma(x),{\mathfrak R}^{-k\vee -k'}+\|\sigma(x)-x_{Q}\|)))^\theta}
        \bigg(\frac {{\mathfrak R}^{-k\vee -k'}}{{\mathfrak R}^{-k\vee -k'}+\|\sigma(x)-x_{Q}\|}\bigg)^{\theta\varepsilon}|q_{k}f(x_{Q})|^\theta\\
        &=\sum\limits_{\ell=0}^\infty\sum\limits_{Q\in A_\ell}w(Q)^\theta \frac1{\omega(B(\sigma(x),{\mathfrak R}^{-k\vee -k'}+\|\sigma(x)-x_{Q}\|))
            )^\theta}
        \bigg(\frac {{\mathfrak R}^{-k\vee -k'}}{{\mathfrak R}^{-k\vee -k'}+\|\sigma(x)-x_{Q}\|}\bigg)^{\theta\varepsilon}|q_{k}f(x_{Q})|^\theta\\
        &\lesssim {\mathfrak R}^{[-k-(-k\vee -k')]{\bf N}(\theta-1)]}\sum\limits_{\ell=0}^\infty \Big(\frac {\omega(B(\sigma(x),{\mathfrak R}^{-k\vee -k'})))}{\omega(B(\sigma(x),{\mathfrak R}^{\ell+(-k\vee -k')})))}
        \Big)^{\theta-1} \frac 1{{\mathfrak R}^{\theta\varepsilon\ell}}\\
        &\qquad\qquad\times \frac1{\omega(B(\sigma(x),{\mathfrak R}^{\ell+(-k\vee -k')})))}\int_{\|\sigma(x)-z\|\leqslant 2{\mathfrak R}^{\ell+(-k\vee -k')}}
        \sum\limits_{Q\in A_\ell}|q_{k}f(x_{Q})|^\theta\chi_{Q}(z)d\omega(z)\\
        &\lesssim {\mathfrak R}^{[-k-(-k\vee -k')]{\bf N}(\theta-1)]}\sum\limits_{\ell=0}^\infty \frac1{{\mathfrak R}^{\ell[\theta\varepsilon+{\bf N}(\theta-1)]}} M\Big(\sum\limits_{Q\in Q^k}|q_{k}f(x_{Q})|^\theta\chi_{Q} \Big)(\sigma(x))\\
        &\lesssim  {\mathfrak R}^{[-k-(-k\vee -k')]{\bf N}(\theta-1)]}M\Big(\sum\limits_{Q\in Q^k}|q_{k}f(x_{Q})|^\theta\chi_{Q} \Big)(\sigma(x)),
        \end{align*}
        where $M$ denote the Hardy-Littlewood maximal operator on $(\Bbb R^N, \|\cdot\|,\omega)$.  Therefore,
        \begin{align*}
        &\sum\limits_{Q\in Q^k}w(Q)\frac1{\omega(B(\sigma(x),{\mathfrak R}^{-k\vee -k'}+\|\sigma(x)-x_{Q}\|))}
        \bigg(\frac {{\mathfrak R}^{-k\vee -k'}}{{\mathfrak R}^{-k\vee -k'}+\|\sigma(x)-x_{Q}\|}\bigg)^{\varepsilon}|q_{k}f(x_{Q})|\\
        &\lesssim {\mathfrak R}^{[-k-(-k\vee -k')]{\bf N}(1-1/\theta)]}
        \bigg\{M\Big(\sum\limits_{Q\in Q^k}|q_{k}f(x_{Q})|^\theta\chi_{Q}\Big)(\sigma(x))\bigg\}^{1/\theta}
        \end{align*}
        and
        \begin{equation}\label{eq 2.9}
        \begin{aligned}
        &\sum\limits_{Q\in Q^k} {\mathfrak R}^{(k'-k)\alpha}\omega(Q)|q_{k'}\psi_{k}(x_{Q'},x_{Q})q_{k}f(x_{Q})|\chi_{Q'}(x) \\
        &\lesssim \sum\limits_{\sigma\in G}{\mathfrak R}^{-|k-k'|(\varepsilon-|\alpha|)}{\mathfrak R}^{[-k-(-k'\vee -k)]{\bf N}(1-\frac 1\theta)}
        \bigg\{M\Big(\sum\limits_{Q\in Q^k}|q_{k}f(x_{Q})|^\theta\chi_{Q}\Big)(\sigma(x))\bigg\}^{1/\theta}\chi_{Q'}(x).
        \end{aligned}
        \end{equation}
        It is clear that for $\frac{\bf N}{{\bf N}+\varepsilon-|\alpha|}<\theta<p,$
        $$\sup_{k'}\sum\limits_{k\in \Bbb Z} {\mathfrak R}^{-|k-k'|(\varepsilon-|\alpha|)}{\mathfrak R}^{[-k-(-k'\vee -k)]{\bf N}(1-\frac 1\theta)}<\infty.$$
        By H\"older's inequality, we have
        \begin{align*}
        &\Big|{\mathfrak R}^{k'\alpha}\sum\limits_{k\in\mathbb Z}\sum\limits_{Q\in Q^k}\omega(Q)
        q_{k'}\psi_{k}(x_{Q'},x_{Q})q_{k}f(x_{Q})\Big|^q\chi_{Q'}(x) \\
        &\lesssim \sum\limits_{\sigma\in G}\sum\limits_{k\in \mathbb Z} {\mathfrak R}^{-|k-k'|(\varepsilon-|\alpha|)}{\mathfrak R}^{[-k-(-k'\vee -k)]{\bf N}(1-\frac 1\theta)}
        \bigg\{{\mathfrak R}^{k\alpha}M\Big(\sum\limits_{Q^k}|\lambda_{Q^k}|^\theta\chi_{Q^k}\Big)(\sigma(x))\bigg\}^{q/\theta}\chi_{Q'}(x).
        \end{align*}
        This implies
        \begin{align*}
        &\Big\{\sum\limits_{k'\in \mathbb Z}\sum\limits_{Q'\in Q^{k'}}
        \Big|{\mathfrak R}^{k'\alpha}\sum\limits_{k\in \mathbb Z}\sum\limits_{Q\in Q^k} \omega(Q)q_{k'}\psi_{k}(x_{Q'},x_{Q})q_{k}f(x_{Q})\Big|^q\chi_{Q'}\Big\}^{1/q} \\
        &\qquad\lesssim \bigg\{\sum\limits_{\sigma\in G}\sum\limits_{k\in \mathbb Z}
        \bigg\{{\mathfrak R}^{k\alpha}M\Big(\sum\limits_{Q\in Q^k}|\lambda_Q|^\theta\chi_{Q}\Big)(\sigma(x))\bigg\}^{q/\theta}\bigg\}^{1/q},
        \end{align*}
        where for $\frac{\bf N}{{\bf N}+\varepsilon-|\alpha|}<\theta<p$, the estimate
         \begin{align}\label{sup finite}
         \sup_k\sum\limits_{k'\in \Bbb Z} {\mathfrak R}^{-|k-k'|(\varepsilon-|\alpha|)}{\mathfrak R}^{[-k-(-k'\vee -k)]{\bf N}(1-\frac 1\theta)}<\infty
          \end{align}
        is used.
        The Fefferman--Stein vector valued maximal function inequality with $\theta<p$ yields
        \begin{align*}
        &\Big\|\Big\{\sum\limits_{k'\in \mathbb Z}\sum\limits_{Q'\in Q^{k'}}
        \Big|{\mathfrak R}^{k'\alpha}\sum\limits_{k\in \mathbb Z}\sum\limits_{Q\in Q^k} \omega(Q)q_{k'}\psi_{k}(x_{Q'},x_{Q})q_{k}f(x_{Q})\Big|^q\chi_{Q'}\Big\}^{1/q}\Big\|_{L^p_\omega} \\
        &\qquad\lesssim \sum\limits_{\sigma\in G}
        \Big\|\Big(\sum\limits_{k\in \mathbb Z}\sum\limits_{Q\in Q^k}|{\mathfrak R}^{k\alpha} q_kf(x_Q)|^q\chi_Q(\sigma(x))\Big)^{\frac{1}{q}}\Big\|_{L^p_\omega}.
        \end{align*}
        Since $G$ is a finite group and
        $$\int_{\Bbb R^N}  f(\sigma(x))d\omega(x)=\int_{\Bbb R^N} f(x)d\omega(x),$$
        we have
        \begin{align*}
        \|T_M(f)\|_{\dot{\mathcal F}^{\alpha,q}_{p,{\rm D}}}
        &\lesssim \Big\|\sum\limits_{k'\in \mathbb Z}\sum\limits_{Q'\in Q^{k'}}
        \Big|\sum\limits_{k\in \mathbb Z}\sum\limits_{Q\in Q^k} \omega(Q)q_{k'}\psi_{k}(x_{Q'},x_{Q})q_{k}f(x_{Q})\Big|^q\chi_{Q'}\Big\|_{L^p_\omega} \\
        &\lesssim
        \Big\|\Big(\sum\limits_{k\in \mathbb Z}\sum\limits_{Q\in Q^k}|q_{k}f(x_{Q})|^q\chi_Q(x)\Big)^{\frac{1}{q}}\Big\|_{L^p_\omega}.
        \end{align*}
For $q=\infty$, we write
\begin{align*}
\|T_M(f)\|_{\dot{\mathcal F}^{\alpha,\infty}_{p,{\rm D}}}
&=\Big\|\sup_{k'\in \mathbb Z}\sum_{Q'\in Q^{k'}} {\mathfrak R}^{k'\alpha}q_{k'}(f)(x_{Q'})\chi
	_Q\Big\|_{L^p_\omega}\\
&\le C\Big\|\sup_{k'\in \mathbb Z}\sum_{Q'\in Q^{k'}}\Big({\mathfrak R}^{k'\alpha}\Big|q_{k'}\Big(\sum_{k\in \Bbb Z}\sum\limits_{Q\in Q^k}w(Q) \psi_{k}(x_{Q'},x_{Q})q_{k}
f(x_{Q})\Big)\Big|\Big)\chi_{Q'}\Big\|_{L^p_\omega}\\
&\le C\Big\|\sup_{k'\in \mathbb Z}\sum_{Q'\in Q^{k'}}\Big({\mathfrak R}^{k'\alpha}\sum_{k\in \Bbb Z}\sum\limits_{Q\in Q^k}w(Q) |q_{k'}\psi_k(x_{Q'},x_Q)q_{k}
f(x_{Q})|\Big)\chi_{Q'}\Big\|_{L^p_\omega}.
\end{align*}
By \eqref{eq 2.9}, we have
\begin{align*}
&{\mathfrak R}^{k'\alpha}\sum_{k\in \Bbb Z}\sum\limits_{Q\in Q^k}w(Q) |q_{k'}\psi_k(x_{Q'},x_Q)q_{k}
f(x_{Q})|\chi_{Q'}(x) \\
& \lesssim \sum\limits_{\sigma\in G}{\mathfrak R}^{-|k-k'|(\varepsilon-|\alpha|)}{\mathfrak R}^{[-k-(-k'\vee -k)]{\bf N}(1-\frac 1\theta)}
        \bigg\{M\Big(\sup_{k\in \mathbb Z}{\mathfrak R}^{k\alpha}\sum\limits_{Q\in Q^k}|q_{k}f(x_{Q})|^\theta\chi_{Q}\Big)(\sigma(x))\bigg\}^{1/\theta}\chi_{Q'}(x).
\end{align*}
Now by using \eqref{sup finite} again we obtain that 

        \begin{align*}
        & \sup_{k'\in \mathbb Z}\sum\limits_{Q'\in Q^{k'}}
        \Big|{\mathfrak R}^{k'\alpha}\sum\limits_{k\in \mathbb Z}\sum\limits_{Q\in Q^k} \omega(Q)q_{k'}\psi_{k}(x_{Q'},x_{Q})q_{k}f(x_{Q})\Big|\chi_{Q'} (x)\\
        &\qquad\lesssim \bigg\{\sum\limits_{\sigma\in G}
        \bigg\{M\Big(\sup_{k\in \mathbb Z}{\mathfrak R}^{k\alpha}\sum\limits_{Q\in Q^k}|\lambda_Q|^\theta\chi_{Q}\Big)(\sigma(x))\bigg\}^{1/\theta}\bigg\},
        \end{align*}
 Another application of the Fefferman-Stein vector valued maximal function inequality with $\theta<p$ yields
        \begin{align*}
        \|T_M(f)\|_{\dot{\mathcal F}^{\alpha,\infty}_{p,{\rm D}}}
        &\lesssim \Big\|\sup_{k'\in \mathbb Z}\sum\limits_{Q'\in Q^{k'}}
        \Big|{\mathfrak R}^{k'\alpha}\sum\limits_{k\in \mathbb Z}\sum\limits_{Q\in Q^k} \omega(Q)q_{k'}\psi_{k}(x_{Q'},x_{Q})q_{k}f(x_{Q})\Big|\chi_{Q'}\Big\|_{L^p_\omega} \\
        &\lesssim \sum\limits_{\sigma\in G}
        \Big\|\Big(\sup_{k\in \mathbb Z}\sum\limits_{Q\in Q^k}|{\mathfrak R}^{k\alpha} q_kf(x_Q)|\chi_Q(\sigma(x))\Big)\Big\|_{L^p_\omega}\\
        &\lesssim
        \Big\|\Big(\sup_{k\in \mathbb Z}\sum\limits_{Q\in Q^k}|q_{k}f(x_{Q})|\chi_Q(x)\Big)\Big\|_{L^p_\omega}.
        \end{align*}
             
Secondly, we are going to show $T_M^{-1}$ is also bounded on $\dot{\mathcal F}^{\alpha,q}_{p,{\rm D}}$. It suffices to show that if $M$ is chosen large enough and ${\mathfrak R}$ is close enough to $1$, for $f\in \dot{\mathcal F}^{\alpha,q}_{p,{\rm D}}$,  there exists a constant $C$ such that
\begin{equation}\label{eq 2.10}
\|R_1(f)+R_M(f)\|_{\dot{\mathcal F}^{\alpha,q}_{p,{\rm D}}}\le C(r-1+{\mathfrak R}^{-M})\|f\|_{\dot{\mathcal F}^{\alpha,q}_{p,{\rm D}}}.
\end{equation}
To this end, we need the Triebel--Lizorkin space defined in the space of
homogeneous type $(\Bbb R^N,\|\cdot\|,\omega)$ in the sense of Coifman and Weiss.  For $|\alpha|<1<{\mathfrak R}, 1\le p< \infty$ and $1\le q\le \infty$, define
$$\dot{\mathcal F}^{\alpha,q}_{p, CW}=\{f\in L^2(\Bbb R^n, \omega) : \|f\|_{\dot{\mathcal F}^{\alpha,q}_{p,CW}}<\infty\},$$
where
$$\|f\|_{\dot{\mathcal F}^{\alpha,q}_{p,CW}}
	=\begin{cases} \displaystyle \Big\|\Big\{\sum_{k\in \mathbb Z}\sum_{Q\in Q^k} \big({\mathfrak R}^{k\alpha}|D_k(f)(x_Q)|\big)^q\chi_Q\Big\}^{1/q}\Big\|_{L^p_\omega}\quad & \text{if}\ \ 1\le q<\infty \\
	\displaystyle \Big\|\sup_{k\in \mathbb Z}\sum_{Q\in Q^k} {\mathfrak R}^{k\alpha}|D_k(f)(x_Q)|\chi
	_Q\Big\|_{L^p_\omega} &\text{if}\ \ q=\infty. \end{cases}$$
We only consider the case $1\le q<\infty$ as the case $q=\infty$ is similar. 
For $f\in \dot{\mathcal F}^{\alpha,q}_{p,CW}$, we use Theorem \ref{dCRh} to get 
\begin{align*}
\|f\|_{\dot{\mathcal F}^{\alpha,q}_{p,{\rm D}}}
&=\bigg\{\sum_{k\in \Bbb Z} \Big({\mathfrak R}^{k\alpha}\Big\|q_{k}\Big(\sum_{j\in \Bbb Z} \sum_{Q\in Q^j}\omega(Q){\widetilde D}_j(x,x_{Q})D_j(f)(x_{Q})\Big)\Big\|_p\Big)^q\bigg\}^{1/q}.
\end{align*}
Note that ${\widetilde D}_j(\cdot ,x_{Q})\in \mathbb M(1,1, {\mathfrak R}^{-j}, x_{Q}).$ We use Lemma \ref{aoe} again to obtain
$$|q_k{\widetilde D}_j(x,x_{Q})|\le C{\mathfrak R}^{-|k-j|\varepsilon}w(Q) \frac1{V(x,x_Q,{\mathfrak R}^{(-j\vee -k)}+d(x,x_Q))}\Big(\frac{{\mathfrak R}^{(-j\vee -k)}}{{\mathfrak R}^{(-j\vee -k)}+\|x-x_Q\|}\Big)^{\varepsilon}.$$
By above method, we have
$$\|f\|_{\dot{\mathcal F}^{\alpha,q}_{p,{\rm D}}}\le C\|f\|_{\dot{\mathcal F}^{\alpha,q}_{p,CW}}.
$$
Conversely, for $f\in \dot{\mathcal F}^{\alpha,q}_{p,{\rm D}}$, we have
\begin{align*}
\|f\|_{\dot{\mathcal F}^{\alpha,q}_{p,CW}}
&=\Big\|\Big\{\sum_{k\in \mathbb Z}\sum_{Q\in Q^k} \big({\mathfrak R}^{k\alpha}|D_k(f)(x_Q)|\big)^q\chi_Q\Big\}^{1/q}\Big\|_{L^p_\omega}\\
%&=\Big\|\Big\{\sum_{k\in \mathbb Z}\sum_{Q\in Q^k} \big({\mathfrak R}^{k\alpha}|D_k(T_Mf+R_1f+R_Mf)(x_Q)|\big)^q\chi_Q\Big\}^{1/q}\Big\|_{L^p_\omega}\\
&\le \Big\|\Big\{\sum_{k\in \mathbb Z}\sum_{Q\in Q^k} \big({\mathfrak R}^{k\alpha}|D_k(T_Mf)(x_Q)|\big)^q\chi_Q\Big\}^{1/q}\Big\|_{L^p_\omega}\\
&\qquad +\Big\|\Big\{\sum_{k\in \mathbb Z}\sum_{Q\in Q^k} \big({\mathfrak R}^{k\alpha}|D_k(R_1f)(x_Q)|\big)^q\chi_Q\Big\}^{1/q}\Big\|_{L^p_\omega}\\
&\qquad +\Big\|\Big\{\sum_{k\in \mathbb Z}\sum_{Q\in Q^k} \big({\mathfrak R}^{k\alpha}|D_k(R_Mf)(x_Q)|\big)^q\chi_Q\Big\}^{1/q}\Big\|_{L^p_\omega}\\
&:=I_1+I_2+I_3.
\end{align*}
We use the same argument to obtain $I_1\le C\|f\|_{\dot{\mathcal F}^{\alpha,q}_{p,{\rm D}}}.$
To estimate $I_2$, we write 
$$R_1f(x)=-\sum\limits_{j=-\infty}^\infty\int_{{\mathfrak R}^{-j}}^{{\mathfrak R}^{-j+1}} \Big[\psi_t\ast q_t\ast f(x)-
\psi_{j}\ast q_{j}\ast
f(x)\Big]\frac{dt}{t}=-\sum\limits_{j=-\infty}^\infty\int_{{\mathfrak R}^{-j}}^{{\mathfrak R}^{-j+1}}S_j(f)(x)\frac{dt}{t},$$
where $S_j(f)(x)=\int_{\mathbb R^N} S_j(x,y)f(y)d\omega(y)$ with
$S_j(x,y)=\psi_{t}q_{ t}(x,y)-\psi_{j}q_{j}(x,y).$ $R_1(x,y),$ the
kernel  of $R_1,$ can be written as
$$R_1(x,y)=-\sum\limits_{j=-\infty}^\infty\int_{{\mathfrak R}^{-j}}^{{\mathfrak R}^{-j+1}} S_j(x,y)\frac{dt}{t}.$$
We have for any $0<\varepsilon<1$,
\begin{itemize}
\item[(1)]    $|R_1(x,y)|\le C({\mathfrak R}-1)\Big(\frac{\displaystyle d(x,y)}{\displaystyle \|x-y\|}\Big)^\varepsilon \frac{\displaystyle 1}{\displaystyle \omega(B(x,d(x,y)))}$;\\[4pt]
\item[(2)] $|R_1(x,y)-R_1(x,y')|\le C({\mathfrak R}-1)\Big(\frac{\displaystyle \|y-y'\|}{\displaystyle \|x-y\|}\Big)^\varepsilon\frac{\displaystyle 1}{\displaystyle \omega(B(x,d(x,y)))}$\  for $\|z-z'\|\le d(x,z)/2;$\\[4pt]
 \item[(3)] $|R_1(x',y)-R_1(x,y)|\le C({\mathfrak R}-1) \Big(\frac{\displaystyle \|x-x'\|}{\displaystyle \|x-y\|}\Big)^\varepsilon\frac{\displaystyle 1}{\displaystyle \omega(B(x,d(x,y)))}$\ for $\|x-x'\|\le d(x,z)/2.$
 \end{itemize}
We also obtain that $R_1$ is a Dunkl--Calder\'on--Zygmund operator with coefficient $(r-1)$ (for details see \cite[pages 54-60]{HHLLT}).  
By Theorem \ref{dCRh} and Lemma \ref{T aoe} and the above arguments,
\begin{align*}
I_2&\le \Big\|\Big\{\sum_{k\in \mathbb Z}\sum_{Q\in Q^k} \big({\mathfrak R}^{k\alpha}\Big|D_k\Big(R_1\Big(\sum\limits_{j\in \mathbb Z}\sum_{Q\in Q^j}\omega(Q)D_j(x,x_{Q})\widetilde{\widetilde{D}}_j(f)(x_{Q}) \Big)\Big
)(x_Q)\Big|\Big)^q\chi_Q\Big\}^{1/q}\Big\|_{L^p_\omega}\\
&\le \Big\|\Big\{\sum_{k\in \mathbb Z}\sum_{Q\in Q^k} \big({\mathfrak R}^{k\alpha}\Big|\sum\limits_{j\in \mathbb Z}\sum_{Q\in Q^j}\omega(Q)D_kR_1D_j(x,x_{Q})\widetilde{\widetilde{D}}_j(f)(x_{Q}) (x_Q)\Big|\Big)^q\chi_Q\Big\}^{1/q}\Big\|_{L^p_\omega}\\
&\le \bigg\{\sum_{k\in \Bbb Z} \Big({\mathfrak R}^{k\alpha}
            \Big\|\sum\limits_{j\in \mathbb Z}\sum_{Q\in Q^j}\omega(Q)D_kR_1D_j(x,x_{Q})\widetilde{\widetilde{D}}_j(f)(x_{Q}) )\Big\|_p\Big)^q\bigg\}^{1/q} \\   
&\le C({\mathfrak R}-1)\|f\|_{\dot{\mathcal F}^{\alpha,q}_{p,CW}},            
\end{align*}
{where the last inequality is done by the same argument as in \cite[Proposition 5,4]{HMY}.}

Similarly, we also have that  $R_M$ is an Dunkl-Calder\'on-Zygmund operator with coefficient ${\mathfrak R}^{-M}$ and
$$I_3\le C{\mathfrak R}^{-M}\|f\|_{\dot{\mathcal F}^{\alpha,q}_{p,CW}}.$$
Choosing $M$ large enough and $\mathfrak R$ greater than 1 and close to 1 such that
$$\|f\|_{\dot{\mathcal F}^{\alpha,q}_{p,CW}}\le C\|f\|_{\dot{\mathcal F}^{\alpha,q}_{p,{\rm D}}}+\frac14 \|f\|_{\dot{\mathcal F}^{\alpha,q}_{p,CW}}+\frac 14\|f\|_{\dot{\mathcal F}^{\alpha,q}_{p, CW}},$$ 
we get
$$\|f\|_{\dot{\mathcal F}^{\alpha,q}_{p,CW}}\le C\|f\|_{\dot{\mathcal F}^{\alpha,q}_{p,{\rm D}}}.$$
Therefore, 
\begin{align*}
\|R_1(f)\|_{\dot{\mathcal F}^{\alpha,q}_{p,{\rm D}}}
&\le C\|R_1(f)\|_{\dot{\mathcal F}^{\alpha,q}_{p,CW}} =C\Big\|\Big\{\sum_{k\in \mathbb Z}\sum_{Q\in Q^k} \big({\mathfrak R}^{k\alpha}|D_k(R_1f)(x_Q)|\big)^q\chi_Q\Big\}^{1/q}\Big\|_{L^p_\omega}\\
&\le C({\mathfrak R}-1)\|f\|_{\dot{\mathcal B}^{\alpha,q}_{p,CW}} \\
&\le C({\mathfrak R}-1)\|f\|_{\dot{\mathcal B}^{\alpha,q}_{p,{\rm D}}}.
\end{align*}
Similarly, $$\|R_M(f)\|_{\dot{\mathcal F}^{\alpha,q}_{p,{\rm D}}}\le C{\mathfrak R}^{-M}\|f\|_{\dot{\mathcal F}^{\alpha,q}_{p,{\rm D}}}.$$
Hence, the inequality \eqref{eq 2.10} holds and the inequality gives
$$\|T_M^{-1}f\|_{\dot{\mathcal F}^{\alpha,q}_{p,{\rm D}}}\le C\|f\|_{\dot{\mathcal F}^{\alpha,q}_{p,{\rm D}}}.$$
Set $h=T_M^{-1}f$. we obtain $f=T_Mh$ 
which implies $\|h\|_{L^2_\omega}\sim \|f\|_{L^2_\omega}$ and $\|h\|_{\dot{\mathcal F}^{\alpha,q}_{p,{\rm D}}}\sim \|f\|_{\dot{\mathcal F}^{\alpha,q}_{p,{\rm D}}}.$
Since 
$$f=T_Mh= -\ln {\mathfrak R}\sum_{j=-\infty}^\infty\sum_{Q\in Q^j}w(Q) \psi_{j}(x,x_{Q})q_{j}
h(x_{Q}).$$
It remains to show that the above series converges in $\dot{\mathcal F}^{\alpha,q}_{p,{\rm D}}$.
We observe that
$$f(x)+\ln {\mathfrak R}\sum_{|j|\le \ell}\sum_{Q\in Q^j}w(Q) \psi_{j}(x,x_{Q})q_{j}
h(x_{Q})=-\ln {\mathfrak R}\sum_{|j|>\ell} \sum_{Q\in Q^j}w(Q) \psi_{j}(x,x_{Q})q_{j}
h(x_{Q})$$
for $f\in L^2(\mathbb R^N, \omega)$. Therefore, we only to show that
$$\lim_{\ell\to \infty} \Big\|\sum_{|j|>\ell} \sum_{Q\in Q^j}w(Q) \psi_{j}(x,x_{Q})q_{j}
h(x_{Q})\Big\|_{\dot{\mathcal F}^{\alpha,q}_{p,{\rm D}}}=0.$$
Since 
\begin{align*}
&\Big\|\sum_{|j|>\ell} \sum_{Q\in Q^j}w(Q) \psi_{j}(x,x_{Q})q_{j}
h(x_{Q})\Big\|_{\dot{\mathcal F}^{\alpha,q}_{p,{\rm D}}} \\
&=C\Big\|\Big\{\sum_{k\in \mathbb Z}\sum_{Q\in Q^k} \Big({\mathfrak R}^{k\alpha}\Big|q_k\Big(\sum_{|j|>\ell} \sum_{Q\in Q^j}w(Q) \psi_{j}(\cdot,x_{Q})q_{j}
h(x_{Q})\Big)(x_Q)\Big|\Big)^q\chi_Q\Big\}^{1/q}\Big\|_{L^p_\omega},
\end{align*}
the same argument as the proof of \eqref{eq 2.9} yields
$$\Big\|\sum_{|j|>\ell} \sum_{Q\in Q^j}w(Q) \psi_{j}(x,x_{Q})q_{j}
h(x_{Q})\Big\|_{\dot{\mathcal F}^{\alpha,q}_{p,{\rm D}}}
\le C\Big\|\Big\{\sum_{|j|>\ell}\sum_{Q\in Q^j} \big({\mathfrak R}^{j\alpha}|q_j(h)(x_Q)|\big)^q\chi_Q\Big\}^{1/q}\Big\|_{L^p_\omega}.$$
Since  $\|h\|_{\dot{\mathcal F}^{\alpha,q}_{p,{\rm D}}}\sim \|f\|_{\dot{\mathcal F}^{\alpha,q}_{p,{\rm D}}}$, the right-hand side of the above inequality goes to $0$ as $\ell\to \infty$.
\end{proof}

\begin{proof}[Proof of Theorem \ref{thm 3.4}]
Given $f\in \dot{\mathcal F}^{\alpha,q}_{p,{\rm D}}$ and $g\in \dot{\mathcal F}^{-\alpha,q'}_{p',{\psi}}$, we use Theorem \ref{thm 3.3} to get
\begin{align*}
|\langle f, g \rangle|
&=\Big| \Big\langle \sum_{j=-\infty}^\infty\sum\limits_{Q\in Q^j}\omega(Q)
    \psi_{{\mathfrak R}^j}(\cdot,x_{Q})q_{{\mathfrak R}^j}
    h(x_{Q}) , g \Big\rangle\Big| \\
&\le  \sum_{j=-\infty}^\infty\sum\limits_{Q\in Q^j}\omega(Q)|q_{{\mathfrak R}^j}h(x_{Q})|
     |\psi_{{\mathfrak R}^j}g(x_Q)| \\
&=\int_{\mathbb R} \sum_{j=-\infty}^\infty\sum\limits_{Q\in Q^j}|q_{{\mathfrak R}^j}h(x_{Q})|
     |\psi_{{\mathfrak R}^j}g(x_Q)|\chi_Q(x)d\omega(x),
\end{align*}
where the second inequality follows since $\psi$ is a radial function. 
By H\"older's inequality applied twice,
\begin{align*}
|\langle f, g \rangle|
&\le \int_{\mathbb R^N}\Big\{\sum_{j=-\infty}^\infty\sum_{Q\in Q^j} ({\mathfrak R}^{j\alpha}|q_{{\mathfrak R}^j}h(x_{Q})|)^q\Big\}^{\frac1q}\Big\{\sum_{j=-\infty}^\infty\sum_{Q\in Q^j} ({\mathfrak R}^{-j\alpha}|\psi_{{\mathfrak R}^j}g(x_Q)|)^{q'}\Big\}^{\frac1{q'}} \chi_Q(x) 
     d\omega(x) \\
&\le \|h\|_{\dot{\mathcal F}^{\alpha,q}_{p,{\rm D}}}\|g\|_{\dot{\mathcal F}^{-\alpha,q'}_{p',\psi}} \\
&\lesssim \|f\|_{\dot{\mathcal F}^{\alpha,q}_{p,{\rm D}}}\|g\|_{\dot{\mathcal F}^{-\alpha,q'}_{p',\psi}}.
\end{align*}
The proof is complete.
\end{proof}

To prove Proposition \ref{thm 1.5}, we establish the following weak-type discrete 
Calder\'on reproducing formula in the distribution sense:
\begin{pro}\label{CRFdis1}
    Suppose that $\{f_n\}_{n=1}^\infty$ is a Cauchy sequence in $\dot{\mathcal F}^{\alpha,q}_{p,{\rm D}}$. Then there exists $f,$ as a distribution in $(\dot{\mathcal F}^{-\alpha,q'}_{p',\psi})^\prime $, such that {\rm (i)}  $\|f\|_{\dot{\mathcal F}^{\alpha,q}_{p,{\rm D}}}=\lim\limits_{n \rightarrow \infty}\|f_n\|_{\dot{\mathcal F}^{\alpha,q}_{p,{\rm D}}}<\infty;$  {\rm (ii)} there exists a distribution $h\in (\dot{\mathcal F}^{-\alpha,q'}_{p',\psi})^\prime$ with $\|f\|_2\sim \|h\|_2, \|f\|_{\dot{\mathcal F}^{\alpha,q}_{p,{\rm D}}}\sim \|h\|_{\dot{\mathcal F}^{\alpha,q}_{p,{\rm D}}},$ such that for each $g\in \dot{\mathcal F}^{-\alpha,q'}_{p',\psi}, f$ has the following weak-type discrete Calder\'on reproducing formula in the distribution sense:
    \begin{align*}
    \langle f,g\rangle&:=\langle -\ln {\mathfrak R}\sum\limits_{j=-\infty}^\infty\sum\limits_{Q\in Q^j}w(Q)\psi_Q(\cdot,x_{Q})q_{Q}h(x_{Q}), g\rangle\\
    &=-\ln {\mathfrak R}\sum\limits_{j=-\infty}^\infty\sum\limits_{Q\in Q^j}w(Q)\psi_Qg(x_{Q})q_{Q}h(x_{Q}),
    \end{align*}
where the last series converges absolutely.
\end{pro}

\begin{proof}
By  Theorem \ref{thm 3.4}, there exists $f\in (\dot{\mathcal F}^{-\alpha,q'}_{p',\psi})^\prime$ such that for each $g\in \dot{\mathcal F}^{-\alpha,q'}_{p',\psi},$
    $$\langle f,g\rangle=\lim\limits_{n \rightarrow \infty} \langle f_n, g\rangle.$$
    Observing that $\|f-f_n\|_{\dot{\mathcal F}^{\alpha,q}_{p,{\rm D}}}=\|\lim\limits_{m \rightarrow \infty}(f_m-f_n))\|_{\dot{\mathcal F}^{\alpha,q}_{p,{\rm D}}} \leqslant \liminf\limits_{m \rightarrow \infty}
    \|f_m-f_n\|_{\dot{\mathcal F}^{\alpha,q}_{p,{\rm D}}},$ and hence, $\|f-f_n\|_{\dot{\mathcal F}^{\alpha,q}_{p,{\rm D}}}\rightarrow 0$ as $n\rightarrow \infty.$ 
    This implies that $\|f\|_{\dot{\mathcal F}^{\alpha,q}_{p,{\rm D}}}=\lim\limits_{n \rightarrow \infty}\|f_n\|_{\dot{\mathcal F}^{\alpha,q}_{p,{\rm D}}}<\infty.$ 
    Applying Theorem \ref{thm 3.3}, for each $f_n$ there exists an $h_n$ such that $\|f_n\|_2\sim \|h_n\|_2$ and $\|f_n\|_{\dot{\mathcal F}^{\alpha,q}_{p,{\rm D}}}\sim \|h_n\|_{\dot{\mathcal F}^{\alpha,q}_{p,{\rm D}}}.$ Thus, by Theorem \ref{thm 3.4}, there exists $h\in (\dot{\mathcal F}^{-\alpha,q'}_{p',\psi})^\prime$ such that for each $g\in \dot{\mathcal F}^{-\alpha,q'}_{p',\psi},$
    $$\langle h,g\rangle=\lim\limits_{n \rightarrow \infty} \langle h_n, g\rangle.$$
    Therefore, $\|h_n-h_m\|_{\dot{\mathcal F}^{\alpha,q}_{p,{\rm D}}}\rightarrow 0$ and 
    $\|h\|_{\dot{\mathcal F}^{\alpha,q}_{p,{\rm D}}}
    =\lim\limits_{n \rightarrow \infty}\|h_n\|_{\dot{\mathcal F}^{\alpha,q}_{p,{\rm D}}}
    \sim \lim\limits_{n \rightarrow \infty} \|f_n\|_{\dot{\mathcal F}^{\alpha,q}_{p,{\rm D}}}
    \sim \|f\|_{\dot{\mathcal F}^{\alpha,q}_{p,{\rm D}}}.$
    
To show that $f$ has a weak-type discrete Calder\'on reproducing formula in the distribution sense, for each $g\in \dot{\mathcal F}^{-\alpha,q'}_{p',\psi},$ applying the proof of the Theorem \ref{thm 3.4},
    $$\bigg|\sum\limits_{j=-\infty}^\infty\sum\limits_{Q\in Q^j}w(Q)\psi_Qg(x_{Q})q_{Q}h(x_{Q})\bigg|
    \le C\|f\|_{\dot{\mathcal F}^{\alpha,q}_{p,{\rm D}}}\|g\|_{\dot{\mathcal F}^{-\alpha,q'}_{p',\psi}},$$
    which implies that the series $\sum\limits_{j=-\infty}^\infty\sum\limits_{Q\in Q^j}w(Q)\psi_Q(x,x_{Q})q_{Q}h(x_{Q})$ is a distribution in $(\dot{\mathcal F}^{-\alpha,q'}_{p',\psi})^\prime.$ Moreover, by the weak-type discrete Calder\'on reproducing formula of $f_n$ in Theorem \ref{thm 3.3}, for each $g\in \dot{\mathcal F}^{-\alpha,q'}_{p',\psi},$
    \begin{eqnarray*}
        \langle f,g\rangle=\lim\limits_{n \rightarrow \infty} \langle f_n, g\rangle
        =\lim\limits_{n \rightarrow \infty}\Big\langle -\ln {\mathfrak R}\sum\limits_{j=-\infty}^\infty\sum\limits_{Q\in Q^j}w(Q)\psi_Q(x,x_{Q})q_{Q}h_n(x_{Q}),
        g\Big\rangle,
    \end{eqnarray*}
    where $\|f_n\|_2\sim \|h_n\|_2$ and $\|f_n\|_{\dot{\mathcal F}^{\alpha,q}_{p,{\rm D}}}\sim \|h_n\|_{\dot{\mathcal F}^{\alpha,q}_{p,{\rm D}}}.$
 
    Observe that, by the same proof of Theorem \ref{thm 3.4},
$$\Big|\Big\langle -\ln {\mathfrak R}\sum\limits_{j=-\infty}^\infty\sum\limits_{Q\in Q^j}w(Q)\psi_Q(x,x_{Q})q_{Q}(h-h_n)(x_{Q}), g\Big\rangle\Big|\leqslant C\|h_n-h\|_{\dot{\mathcal F}^{\alpha,q}_{p,{\rm D}}}\|g\|_{\dot{\mathcal F}^{-\alpha,q'}_{p',\psi}},$$
    where the last term above tends to zero as $n\rightarrow \infty$ and hence,
    $$\langle f,g\rangle=\lim\limits_{n \rightarrow \infty}\langle f_n,g\rangle=\Big\langle -\ln {\mathfrak R}\sum\limits_{j=-\infty}^\infty\sum\limits_{Q\in Q^j}w(Q)\psi_Q(x,x_{Q})q_{Q}
    h(x_{Q}), g\Big\rangle.$$
    The proof of Proposition \ref{CRFdis1} is complete.    
\end{proof}

We now show Proposition \ref{thm 1.5}.

\begin{proof}[Proof of Proposition \ref{thm 1.5}]
Suppose $f\in \dot F^{\alpha,q}_{p,{\rm D}}.$ Then $f\in (\dot{\mathcal F}^{-\alpha,q'}_{p',\psi})^\prime$ and $f$ has a wavelet-type decomposition $f(x)=\sum\limits_{j=-\infty}^\infty\sum\limits_{Q\in Q^j}w(Q)\lambda_{Q}\psi_Q(x,x_{Q})$ in $(\dot{\mathcal F}^{-\alpha,q'}_{p',\psi})^\prime$ with $$ \Big\|\Big\{\sum\limits_{j}\sum\limits_{Q\in Q^j}|\lambda_{Q}|^q\chi_{Q}\Big\}^{\frac{1}{q}}\Big\|_{L^p_\omega}<\infty.$$  Set
    $$f_n(x)=\sum\limits_{|j|\leqslant n}\sum\limits_{\substack{Q\in Q_j\\ Q \subseteq B(0, n)} }w(Q)\lambda_{Q}\psi_Q(x,x_{Q}).$$
    Then $f_n\in \dot{\mathcal F}^{\alpha,q}_{p,{\rm D}}$ and $f_n$ converges  to $f$ in $(\dot{\mathcal F}^{-\alpha,q'}_{p',\psi})^\prime$ as $n$ tends to $\infty.$ To see that $f\in \overline{\dot{\mathcal F}^{\alpha,q}_{p,{\rm D}}},$ by Proposition \ref{CRFdis1}, it suffices to show that  $\|f_n-f_m\|_{\dot{\mathcal F}^{\alpha,q}_{p,{\rm D}}}\rightarrow 0$ as $n,m\rightarrow \infty.$ Indeed, if let $E_n=\{(j, Q): |j|\le n, Q\in Q^j\subseteq B(0,n)\}$ and $E^c_{n,m}=E_n/E_m$ with $n\ge m,$
   \begin{align*}
    \|f_n-f_m\|_{\dot{\mathcal F}^{\alpha,q}_{p,{\rm D}}}&= \Big\|\Big(\sum\limits_{k\in \Bbb Z}\sum\limits_{Q'\in Q^k}|q_{Q'}(f_n-f_m)(x_{Q'})|^q\chi_{Q'}(x)\Big)^{\frac{1}{q}}\Big\|_{L^p_\omega}\\
    &\le \Big\|\Big(\sum\limits_{k\in \Bbb Z}\sum\limits_{Q'\in Q^k}|q_{Q'}\big(\sum\limits_{E^c_{n,m}}w(Q)\lambda_{Q}\psi_Q(\cdot,x_{Q})  \big)(x_{Q'})|^q\chi_{Q'}(x)\Big)^{\frac{1}{q}}\Big\|_{L^p_\omega}\\
    &\le C\Big\|\{\sum\limits_{E^c_{n,m}}|\lambda_{Q}|^q\chi_{Q}\}^{\frac{1}{q}}\Big\|_{L^p_\omega}\rightarrow 0,
    \end{align*}
    as $ n, m$ tend to $\infty,$ where the last inequality follows from the same proof of  Theorem \ref{thm 3.3} and hence, $f\in \overline{\dot{\mathcal F}^{\alpha,q}_{p,{\rm D}}}.$ 
    
 Conversely, if $f\in \overline{\dot{\mathcal F}^{\alpha,q}_{p,{\rm D}}}$ by Proposition \ref{CRFdis1}, then there exists $h\in (\dot{\mathcal F}^{-\alpha,q'}_{p',\psi})^\prime$ with $\|h\|_{\dot{\mathcal F}^{\alpha,q}_{p,{\rm D}}}\sim \|f\|_{\dot{\mathcal F}^{\alpha,q}_{p,{\rm D}}}$ such that for each $g\in \dot{\mathcal F}^{-\alpha,q'}_{p',\psi},$
    $$\langle f,g\rangle=\bigg\langle -\ln {\mathfrak R}\sum\limits_{j=-\infty}^\infty\sum\limits_{Q\in Q^j}w(Q)\psi_Q(x,x_{Q})q_{Q}h(x_{Q}), g\bigg\rangle.$$
    Set $\lambda_Q=-\ln {\mathfrak R}\ q_{Q}h(x_{Q})$ with $Q\in Q^j.$ We obtain a wavelet-type decomposition of $f$ in $(\dot{\mathcal F}^{-\alpha,q'}_{p',\psi})^\prime$ in the distribution sense:
    $$f=\sum\limits_{j=-\infty}^\infty\sum\limits_{Q\in Q^j}w(Q)\lambda_Q\psi_Q(x,x_{Q})$$
    and hence, $f\in \dot{F}^{\alpha,q}_{p,{\rm D}}.$ Moreover
    $$\|f\|_{\dot{\mathcal F}^{\alpha,q}_{p,{\rm D}}}=\inf \Big\{\Big\|\{\sum\limits_{j=-\infty}^\infty\sum\limits_{Q\in Q^j}|\lambda_{Q}|^q\chi_{Q}\}^{\frac{1}{q}}\Big\|_{L^p_\omega}\Big\}\le C\|h\|_{\dot{\mathcal F}^{\alpha,q}_{p,{\rm D}}}\leqslant C\|f\|_{\dot{\mathcal F}^{\alpha,q}_{p,{\rm D}}}.$$
The proof of Proposition \ref{thm 1.5} is complete.
\end{proof}

The relationship between the Dunkl--Triebel--Lizorkin space ${\dot F}^{\alpha,q}_{p,{\rm D}}$ and the Triebel--Lizorkin space ${\dot F}^{\alpha,q}_{p,CW}$ on a space of homogeneous type $(\Bbb R^N, \|\cdot\|, \omega)$ in the sense of Coifman and Weiss is given by the following result.

\begin{thm}\label{thm 3.12}
 Let $|\alpha|<1$, $1< p<\infty$ and $1\leq q\leq\infty$.
  The  Dunkl--Triebel--Lizorkin space ${\dot F}^{\alpha,q}_{p,{\rm D}}$  is equivalent to the  Triebel--Lizorkin space ${\dot F}^{\alpha,q}_{p,CW}(\mathbb R^N, \|\cdot\|, \omega)$ in the sense that if $f\in {\dot F}^{\alpha,q}_{p,{\rm D}}$ then $f\in {\dot F}^{\alpha,q}_{p,CW}$ and there exists a constant $C$ such that $\|f\|_{{\dot F}^{\alpha,q}_{p,CW}}\leqslant C\|f\|_{{\dot F}^{\alpha,q}_{p,{\rm D}}}.$
    Conversely, if $f\in {\dot F}^{\alpha,q}_{p,CW}(\mathbb R^N, \|\cdot\|, \omega)$ then $f$ can extend to a distribution $ \widetilde{f}$ on $(\dot{\mathcal F}^{-\alpha,q'}_{p',{\psi}})^\prime$ such that $\langle\widetilde{f}, g\rangle=\langle f,g \rangle$ for all $g\in \mathcal M(\beta,\gamma,r,x_0))^\prime$ and $\widetilde{f}\in {\dot F}^{\alpha,q}_{p,{\rm D}},$ 
 Moreover, $\|\widetilde{f}\|_{{\dot{\mathcal F}}^{\alpha,q}_{p,{\rm D}}}\le C\|f\|_{{\dot F}^{\alpha,q}_{p,CW}}.$
\end{thm}

\begin{proof}
    The proof of Theorem \ref{thm 3.12} is based on Theorem \ref{thm 3.3} and Proposition \ref{thm 1.5}. Indeed, for $f\in {\dot{\mathcal F}}^{\alpha,q}_{p,{\rm D}},$ by Theorem \ref{thm 3.3}, 
    $\|f\|_{{\dot{\mathcal F}}^{\alpha,q}_{p,{\rm D}}}\sim \|f\|_{{\dot F}^{\alpha,q}_{p,CW}}.$ 
   Therefore $\overline{{\dot{\mathcal F}}^{\alpha,q}_{p,{\rm D}}}=\overline{{\dot F}^{\alpha,q}_{p,CW}}$ with the equivalent norms. Given $f\in {\dot{F}}^{\alpha,q}_{p,{\rm D}},$ by Theorem \ref{thm 3.3}, there exists a sequence $\{f_n\}_{n=1}^\infty$ such that for each $f_n\in {\dot{\mathcal F}}^{\alpha,q}_{p,{\rm D}}$ and $f_n$ converges to $f$ in $(\dot{\mathcal F}^{-\alpha,q'}_{p',{\psi}})^\prime.$ 
 Moreover, $\|f\|_{\dot{\mathcal F}^{\alpha,q}_{p,{\rm D}}}=\lim\limits_{n\rightarrow \infty}=\|f_n\|_{\dot{\mathcal F}^{\alpha,q}_{p,{\rm D}}}.$ 
 It is well known that by a classical result, $L^2(\mathbb R^N,\omega)\cap {\dot F}^{\alpha,q}_{p,CW}$ is dense in ${\dot F}^{\alpha,q}_{p,CW}$ and by Lemma \ref{aoe},  $(\dot{\mathcal F}^{-\alpha,q'}_{p',{\psi}})^\prime\subseteq (\mathcal M(\beta,\gamma,r,x_0))^\prime.$ Hence, $f_n$ also converges to $f$ in $(\mathcal M(\beta,\gamma,r,x_0))^\prime$ and $\|f\|_{{\dot F}^{\alpha,q}_{p,CW}} =\lim\limits_{n\rightarrow \infty}\|f_n\|_{{\dot F}^{\alpha,q}_{p,CW}}.$ This implies that
     $$\|f\|_{{\dot F}^{\alpha,q}_{p,CW}}\le C\|f\|_{{\dot{\mathcal F}}^{\alpha,q}_{p,{\rm D}}}.$$
    Suppose $f\in {\dot F}^{\alpha,q}_{p,CW}.$ there exists a sequence $\{f_n\}\in L^2(\mathbb R^N, \omega)$ such that $f_n$ converges $f$ in $(\mathcal M(\beta,\gamma,r,x_0))^\prime$ and $\|f\|_{{\dot F}^{\alpha,q}_{p,CW}}=\lim\limits_{n\rightarrow \infty}|f_n\|_{{\dot F}^{\alpha,q}_{p,CW}}.$ By the proof of Theorem \ref{thm 3.3}, $\|f_n-f_m\|_{{\dot{\mathcal F}}^{\alpha,q}_{p,{\rm D}}}\sim \|f_n-f_m\|_{{\dot F}^{\alpha,q}_{p,CW}} $ and hence, $\|f_n-f_m\|_{{\dot{\mathcal F}}^{\alpha,q}_{p,{\rm D}}}=\|f_n-f_m\|_{{\dot{\mathcal F}}^{\alpha,q}_{p,{\rm D}}}$ tends to zero as $n, m$ tends to $\infty.$
    Therefore, by Proposition \ref{CRFdis1}, $f_n$ tends to $\widetilde{f}$ in $(L^2(\mathbb R^N,\omega)\cap CMO_d^p)^\prime.$ It is clear that $ \widetilde{f}=f$ in $(\mathcal M(\beta,\gamma,r,x_0))^\prime.$ Moreover, by Proposition \ref{thm 1.5},
    $$\|\widetilde{f}\|_{{\dot{\mathcal F}}^{\alpha,q}_{p,{\rm D}}}=\lim\limits_{n\rightarrow \infty}\|f_n\|_{{\dot{\mathcal F}}^{\alpha,q}_{p,{\rm D}}}\le C\lim\limits_{n\rightarrow \infty}\|f_n\|_{{\dot F}^{\alpha,q}_{p,CW}}=C\|f\|_{{\dot F}^{\alpha,q}_{p,CW}}.$$
The proof of Theorem \ref{thm 3.12} is complete.
\end{proof}

%%%%%%%%%%%%%%%%%%%%%%%%%%%%%%%
%%%%%%%%%%%%%%%%%%%%%%%%%%%
%%%%%%%%%%%%%%%%%%%%%%%%%%%%%%%

\section{$\Lambda^\beta$ and $\Lambda_d^\beta$ ($0<\beta<1$) and the commutator of Dunkl Riesz transform: Proof of Theorem \ref{thm commutator Riesz}}
In above section, we show that  the  Dunkl--Triebel--Lizorkin space ${\dot F}^{\alpha,q}_{p,{\rm D}}$  is equivalent to the Triebel--Lizorkin space ${\dot F}^{\alpha,q}_{p,CW}(\mathbb R^N, \|\cdot\|, \omega)$.
To prove Theorem \ref{thm commutator Riesz}, we need the difference characterization of
Triebel--Lizorkin spaces on spaces of homogeneous type see \cite[Propositions 4.1 and 4.6]{WHYY}.

\begin{pro}\label{defference}
For $0<\beta<1<p<\infty$, we have
$$\|f\|_{{\dot F}^{\beta,\infty}_{p,{\rm D}}}\sim \bigg\|\sup_{k\in \mathbb Z} \frac 1{\omega(B(\cdot, {\mathfrak R}^k))
{{\mathfrak R}^{k\beta}}}\int_{B(\cdot, {\mathfrak R}^k)} |f(\cdot)-f(y)|d\omega(y)\bigg\|_{L^p_\omega}.$$ 
\end{pro}

Using the above result, we have the following theorem.

\begin{thm}\label{thm 4.2}
For $0<\beta<1$ and $1<p<\infty$, we have
$$\|f\|_{{\dot F}^{\beta,\infty}_{p,{\rm D}}}\sim \bigg\|\sup_{k\in \mathbb Z} \frac 1{\omega(B(\cdot, {\mathfrak R}^k)){{\mathfrak R}^{k\beta}}}\int_{B(\cdot, {\mathfrak R}^k)} |f(y)-f_{B(\cdot,{\mathfrak R}^k)}|d\omega(y)\bigg\|_{L^p_\omega}.$$ 
\end{thm}

\begin{proof}
Fix $x\in \mathbb R^N$ and $k\in \mathbb Z$.
\begin{align*}
 &\frac 1{\omega(B(\cdot, {\mathfrak R}^k)){ {\mathfrak R}^{k\beta}}}\int_{B(\cdot, {\mathfrak R}^k)} |f(y)-f_{B(\cdot,{\mathfrak R}^k)}|d\omega(y)\\
 &\quad \le \frac1{\omega(B(\cdot, {\mathfrak R}^k)){ {\mathfrak R}^{k\beta}}}\int_{B(\cdot, {\mathfrak R}^k)} |f(y)-f(x)|d\omega(y) \\
      &\qquad\qquad + \frac1{\omega(B(\cdot, {\mathfrak R}^k)){ {\mathfrak R}^{k\beta}}}\int_{B(\cdot, {\mathfrak R}^k)} |f(x)-f_{B(\cdot,{\mathfrak R}^k)}|d\omega(y) \\
  &\quad \le \frac1{\omega(B(\cdot, {\mathfrak R}^k)){ {\mathfrak R}^{k\beta}}}\int_{B(\cdot, {\mathfrak R}^k)} |f(y)-f(x)|d\omega(y) \\
      &\qquad\qquad + \frac1{\omega(B(\cdot, {\mathfrak R}^k)){ {\mathfrak R}^{k\beta}}} \frac1{\omega(B(\cdot, {\mathfrak R}^k))}\int_{B(\cdot, {\mathfrak R}^k)}\int_{B(\cdot, {\mathfrak R}^k)}  |f(x)-f(z)|d\omega(z)d\omega(y) \\
   &\quad \le \frac2{\omega(B(\cdot, {\mathfrak R}^k)){ {\mathfrak R}^{k\beta}}}\int_{B(\cdot, {\mathfrak R}^k)} |f(y)-f(x)|d\omega(y).
  \end{align*}
Hence,
\begin{align*}
& \bigg\|\sup_{k\in \mathbb Z} \frac 1{\omega(B(\cdot, {\mathfrak R}^k)){ {\mathfrak R}^{k\beta}}}\int_{B(\cdot, {\mathfrak R}^k)} |f(y)-f_{B(\cdot,{\mathfrak R}^k)}|d\omega(y)\bigg\|_{L^p_\omega} \\
&\quad \le 2 \bigg\|\sup_{k\in \mathbb Z} \frac 1{\omega(B(\cdot, {\mathfrak R}^k)){ {\mathfrak R}^{k\beta}}}\int_{B(\cdot, {\mathfrak R}^k)} |f(\cdot)-f(y)|d\omega(y)\bigg\|_{L^p_\omega}.
\end{align*}

Conversely, we use the proof of \cite[Propositions 4.1]{WHYY} by replacing $f$ by $f_B$ and get the following
estimate
$$\|f\|_{{\dot F}^{\alpha,\infty}_{p,{\rm D}}}\lesssim \bigg\|\sup_{k\in \mathbb Z} \frac 1{\omega(B(\cdot, {\mathfrak R}^k))
{ {\mathfrak R}^{k\beta}}}\int_{B(\cdot, {\mathfrak R}^k)} |f(\cdot)-f(y)|d\omega(y)\bigg\|_{L^p_\omega}.$$
Hence, Proposition \ref{defference} gives
\begin{align*}
& \bigg\|\sup_{k\in \mathbb Z} \frac 1{\omega(B(\cdot, {\mathfrak R}^k)){ {\mathfrak R}^{k\beta}}}\int_{B(\cdot, {\mathfrak R}^k)} |f(\cdot)-f(y)|d\omega(y)\bigg\|_{L^p_\omega}\\
&\lesssim \bigg\|\sup_{k\in \mathbb Z} \frac 1{\omega(B(\cdot, {\mathfrak R}^k))
{ {\mathfrak R}^{k\beta}}}\int_{B(\cdot, {\mathfrak R}^k)} |f(\cdot)-f(y)|d\omega(y)\bigg\|_{L^p_\omega},
\end{align*}
which completes the proof.
 \end{proof}

The definition of $\Lambda^\beta$ gives the following Lemma.

\begin{lem}\label{lem 4.1}
There is a constant $C>0$ such that for all $b\in \Lambda^\beta$, $r_1>r>0$, $x,y\in \mathbb R^N,$
and $\sigma\in G$, we have:
\begin{align}
|b_{B(x,r)}-b_{B(x,r_1)}|&\le Cr_1^\beta\|b\|_{\Lambda^\beta}, \label{eq 4.1}\\
|b_{B(x,r)}-b_{B(y,r_1)}|&\le Cr^\beta\|b\|_{\Lambda^\beta}\qquad\mbox{\rm for }\|x-y\|\le2r, \label{eq 4.2}\\
|b_{B(x,r)}-b_{B(\sigma(x),r)}|&\le C(\|\sigma(x)-x\|+r)^\beta\|b\|_{\Lambda^\beta}. \label{eq 4.3}
\end{align}
\end{lem}
\begin{proof}
Let $b\in \Lambda^\beta$. Then
\begin{align*}
|b_{B(x,r)}-b_{B(x,r_1)}|&\le \frac1{\omega(B(x,r))}\int_{B(x,r)}|b(y)-b_{B(x,r_1)}|d\omega(y)\\ 
                                      &\le \frac1{\omega(B(x,r))}\frac1{\omega(B(x,r_1))}\int_{B(x,r)}\int_{B(x,r_1)}|b(y)-b(z)|d\omega(y) d\omega(y)  \\
                                      &\le (2r_1)^\beta \|b\|_{\Lambda^\beta}.  
\end{align*}
The inequality \eqref{eq 4.1} is done. If $\|x-y\|\le2r$, then the balls $B(x,r)$ and $B(y,r)$ are contained in 
$B(x,5r)$. The same analysis give the inequality \eqref{eq 4.2} . In order to prove \eqref{eq 4.3}, note that
if $\|x-\sigma(x)\|\le 2r$, then \eqref{eq 4.3} follows by \eqref{eq 4.2}. Assume that $\|x-\sigma(x)\|>2r$ and
let $j$ be the smallest positive integer such that $\|x-\sigma(x)\|\le 2^jr.$ Then applying  \eqref{eq 4.1} and  \eqref{eq 4.2}, we get
\begin{align*}
&|b_{B(x,r)}-b_{B(\sigma(x),r)}|\\
&\le |b_{B(x,r)}-b_{B(x,2^{j+2}r)}|+|b_{B(x,2^{j+2}r)}-b_{B(\sigma(x),2^{j+2}r)}|+|b_{B(\sigma(x),2^{j+2}r)}-b_{B(\sigma(x),r)}|\\
&\le C(2^{j+2}r)^\beta \|b\|_{\Lambda^\beta}\le C(\|\sigma(x)-x\|)^\beta\|b\|_{\Lambda^\beta}\\
&\le C(\|\sigma(x)-x\|+r)^\beta\|b\|_{\Lambda^\beta}.
\end{align*}
The proof is complete.
\end{proof}
We now prove Theorem \ref{thm commutator Riesz}. We note that to prove the sufficiency condition, we use the recent new method in \cite[Theorem 3.1]{DH3}.

\begin{proof}[Proof of Theorem \ref{thm commutator Riesz}]
Let $1<p<\infty$. We choose the index $t$ such that $1<t<p$ and then choose the index $s$ such that 
$s>1$ and that $st<p$. 

We shall prove 
$$\|[b,R_{{\rm D},\ell}]f\|_{{\dot F}^{\alpha,\infty}_{p,{\rm D}}}\lesssim \|b\|_{\Lambda^\beta}\|f\|_{L^p_\omega}$$
for compactly supported functions $f\in  L^p_\omega$ which form a dense subspace in $L^p_\omega$.
Fix $x\in \mathbb R^N, k\in \mathbb Z$ and set $B=B(x, {\mathfrak R}^k)$.
We write all the elements of $G\setminus\{\rm id\}$ in a sequence $\sigma_1,\sigma_2,\ldots,\sigma_{|G|-1}$.
We define the sets $U_j\subseteq \mathbb R^N, j=1,2,\ldots, |G|-1$, inductively: 
\begin{align*}
&U_1=\{z\in \mathbb R^N : \|z-x\|>5{\mathfrak R}^k, \|z-\sigma_1(x)\|\le 5{\mathfrak R}^k\},\\
&U_{j+1}=\{z\in \mathbb R^N :\|z-x\|>5{\mathfrak R}^k, \|z-\sigma_{j+1}(x)\|\le 5{\mathfrak R}^k\}\setminus
               \Big(\bigcup_{i=1}^j U_i\Big)\quad\mbox{for } 1\leq j\leq |G|-1.
\end{align*}    

\noindent For a  compactly supported functions $f\in  L^p_\omega$, we decompose
\begin{equation}\label{eq decom}
f=f_1+f_2+\sum_{j=1}^{|G|-1}f_{\sigma_j},\quad\mbox{where } f_1=f\chi_{5B},\ f_2= f\chi_{({\mathcal O} (5B))^c},
\ f_{\sigma_j}=f\chi_{U_j}.
\end{equation}   
Observe that $[b,R_{{\rm D},\ell}]f=[b-b_B,R_{{\rm D},\ell}]f$.
For $y\in \mathbb R^N$ we set
\begin{align*} 
&[b,R_{{\rm D},\ell}]f_1(y)=(b(y)-b_B)R_{{\rm D},\ell}f_1(y)+R_{{\rm D},\ell}((b_B-b)f_1)(y)
=: g_{11}(y)+g_{12}(y),\\
&[b,R_{{\rm D},\ell}]f_2(y)=(b(y)-b_B)R_{{\rm D},\ell}f_2(y)+R_{{\rm D},\ell}((b_B-b)f_2)(y) =: g_{21}(y)+g_{22}(y),\\
&[b,R_{{\rm D},\ell}]f_{\sigma_j}(y)=(b(y)-b_B)R_{{\rm D},\ell}f_{\sigma_j}(y)+R_{{\rm D},\ell}((b_B-b)f_{\sigma_j})(y) =: g_{\sigma_j1}(y)+g_{\sigma_j2}(y).
\end{align*}      
Note that $|(g_{11})_B|\le \frac1{\omega(B)}\int_B |g_{11}(y)|d\omega(y)$.
Let $t'$ be the conjugate of the index $t$ chosen as the beginning. H\"older's inequality gives
\begin{align*}
&\frac1{\omega(B) \ell(B)^\beta}\int_B |g_{11}(y)-(g_{11})_B|d\omega(y)\le \frac2{\omega(B) \ell(B)^\beta}\int_B |g_{11}(y)|d\omega(y)\\
&\qquad\le \frac2{\omega(B) \ell(B)^\beta}\int_B |(b(y)-b_B)R_{{\rm D},\ell}f_1(y)|d\omega(y)\\
&\qquad\le \frac2{\ell(B)^\beta} \Big(\frac1{\omega(B)} \int_B |b(y)-b_B|^{t'} d\omega(y)\Big)^{\frac1{t'}} 
           \Big(\frac1{\omega(B)}\int_B |R_{{\rm D},\ell}f_1(y)|^td\omega(y)\Big)^{\frac 1t}\\ 
&\qquad\le C\|b\|_{\Lambda^\beta}M_t(R_{{\rm D},\ell}f_1)(x),
\end{align*}
where the last inequality follows from Theorem \ref{lip and lipd} and $M_t(f)=(M(f^t))^{\frac1t}$ with $M$  the Hardy--Littlewood maximal function in the space of
homogeneous type $(\Bbb R^N,\|\cdot\|,\omega)$. Similarly, we also have
\begin{align*}
&\frac1{\omega(B) \ell(B)^\beta}\int_B |g_{21}(y)-(g_{21})_B|d\omega(y)+\frac1{\omega(B) \ell(B)^\beta}\int_B |g_{\sigma_j1}(y)-(g_{\sigma_j1})_B|d\omega(y)\\
&\le C\|b\|_{\Lambda^\beta}\big(M_t(R_{{\rm D},\ell}f_1)(x)+M_t(R_{{\rm D},\ell}f_{\sigma_j})(x)\big)\qquad\mbox{for }j=1,2,\cdots,|G|-1.
\end{align*}
To deal with $g_{12}$, we use H\"older's inequality and then the $L^t_\omega$-boundedness of $R_{{\rm D},\ell}$  to obtain that
\begin{align*}
&\frac1{\omega(B) \ell(B)^\beta}\int_B |g_{12}(y)-(g_{12})_B|d\omega(y)\\
&\qquad\le \frac2{\omega(B) \ell(B)^\beta}\int_B |g_{12}(y)|d\omega(y)\\
&\qquad=\frac2{\omega(B) \ell(B)^\beta}\int_B |R_{{\rm D},\ell}((b-b_B)f_1)(y)|d\omega(y) \\
&\qquad\le \frac{2}{\omega(B) \ell(B)^\beta}\|R_{{\rm D},\ell}((b-b_B)f_1)\|_{L^t_\omega}\omega(B)^{1-\frac 1t} \\
&\qquad\le {C}\omega(B)^{-\frac 1t}\ell(B)^{-\beta} \|(b-b_B)f_1\|_{L^t_\omega}.
\end{align*}
Let $s'$ be the conjugate of the index $s$ chosen as the beginning. By H\"older's inequality, we have
\begin{align*}
\frac1{\omega(B)}\|(b-b_B)f_1\|^t_{L^t_\omega}
&=\frac1{\omega(B)}\int_{5B} |b(y)-b_B|^t|f|^td\omega(y) \\
&\le {  \Big(\frac1{\omega(B)}\int_{5B} |b(y)-b_B|^{ts'}d\omega(y) \Big)^{\frac1{s'}}}
         \Big(\frac1{\omega(B)}\int_{5B} |f|^{ts} d\omega(y)\Big)^{\frac1s}.
\end{align*}
We use Theorem \ref{lip and lipd} to get
\begin{equation}\label{norm-d}
\begin{aligned}
&\frac1{\omega(B)}\int_{5B} |b(y)-b_B|^{ts'}d\omega(y) \\
&\le 2^{ts'}\frac1{\omega(B)}\int_{5B} |b(y)-b_{5B}|^{ts'}d\omega(y) 
       + 2^{ts'} |b_{5B}-b_B|^{ts'}\\
&\lesssim \frac1{\omega(5B)} \int_{5B} |b(y)-b_{5B}|^{ts'}d\omega(y) 
     + \frac1{\omega(5B)}\int_{5B} |b_{5B}-b(y)|^{ts'}d\omega(y) \\
&\lesssim \ell(B)^{\beta ts'}\|b\|^{ts'}_{\Lambda^\beta}.
\end{aligned}
\end{equation}

Thus, $$\frac1{\omega(B) \ell(B)^\beta}\int_B |g_{12}(y)-(g_{12})_B|d\omega(y)\le C\|b\|_{\Lambda^\beta}M_{ts}(f)(x).$$
We turn to  analyze $g_{22}$. Observe that for $ z\notin  {\mathcal O}(5B)$ and $y\in B$ we have $\|x-y\|\le d(x,z)/2$.
Let $\Gamma$ be a fixed closed Weyl chamber such that $x\in \Gamma$, then 
\begin{align*}
|g_{22}(y)-g_{22}(x)|
&=\bigg|\int_{\mathbb R^N} (K_{{\rm D},\ell}(y,z)-K_{{\rm D},\ell}(x,z))(b(z)-b_B)f_2(z)d\omega(z) \bigg|\\
&\le \sum_{\sigma\in G}\int_{\sigma(\Gamma)} |K_{{\rm D},\ell}(y,z)-K_{{\rm D},\ell}(x,z)||b(z)-b_B||f_2(z)|d\omega(z)\\
&\le C \sum_{\sigma\in G}\int_{\sigma(\Gamma)} \frac{\|y-x\|}{\|x-z\|}
          \frac1{\omega(B(x,d(x,z)))}|b(z)-b_B| |f_2(z)|d\omega(z)\\
&=: \sum_{\sigma\in G} J_\sigma(x,y),
\end{align*}
where $K_{D,\ell}$ is the kernel associated to $R_{D,\ell}$ and the second inequality follows from the pointwise regularity estimates for $K_{D,\ell}$ investigated in \cite[Theorem 1.1]{HLLW}.
Note that $$\mathcal{O}(B(x,r))=\bigcup_{\sigma\in G}B(\sigma(x),r)=\{y\in \mathbb{R}^{N}:d(x,y)<r\}$$
and $$\omega(B(x,r))\leq \omega\big(\mathcal{O}(B(x,r))\big)\leq |G|\omega(B(x,r))$$
(see for example \cite{ADH, DH2}). We have
$$\omega(B(x,r)) \sim \omega(B(\sigma(x),r))\qquad\mbox{for }\sigma\in G.$$
In dealing with $J_\sigma(x,y)$ we shall use the inequalities:
\begin{align*}
& \|x-z\|\ge \max(\|x-\sigma(x)\|/2, {\mathfrak R}^k)\quad\mbox{for }z\in \sigma(\Gamma),\\
&5r\le \|\sigma(x)-z\|=d(x,z)\le \|x-z\|\quad\mbox{for }z\in \sigma(\Gamma), z\notin {\mathcal O}(5B).
\end{align*}
Hence,
\begin{align*}
J_\sigma(x,y)&\le C\int_{\sigma(\Gamma)} \frac{{\mathfrak R}^k}{\|x-z\|}
          \frac1{\omega(B(\sigma(x),d(x,z)))}|b_{B(\sigma(x),{\mathfrak R}^k)}-b_B| |f_2(z)|d\omega(z)\\
          &\qquad+C\int_{\sigma(\Gamma)} \frac{{\mathfrak R}^k}{\|x-z\|}
          \frac1{\omega(B(\sigma(x),d(x,z)))}|b_{B(\sigma(x),{\mathfrak R}^k)}-b(z)| |f_2(z)|d\omega(z)\\
          &=:J_{\sigma,1}(x,y)+J_{\sigma,2}(x,y).
\end{align*}
By \eqref{eq 4.3}, 
\begin{equation}\label{eq 4.6}
\begin{aligned}
&J_{\sigma,1}(x,y)\\
&\le \int_{\sigma(\Gamma)}\Big(\frac{{\mathfrak R}^k}{\|x-z\|}\Big)^{1-\beta}\Big(\frac{{\mathfrak R}^k}{{\mathfrak R}^k+\|x-\sigma(x)\|}\Big)^\beta\Big(\frac{\|x-\sigma(x)\|}{{\mathfrak R}^k}+1\Big)^\beta{\mathfrak R}^{k\beta}\|b\|_{\Lambda^\beta}  \\
&\qquad\times \frac1{\omega(B(\sigma(x), \|\sigma(x)-z\|))} |f_2(z)|d\omega(z)\\
&\lesssim \ell(B)^\beta\|b\|_{\Lambda^\beta} \int_{\sigma(\Gamma)} \Big(\frac{{\mathfrak R}^k}{\|x-z\|}\Big)^{1-\beta} \frac{ |f_2(z)|}{\omega(B(\sigma(x), \|\sigma(x)-z\|))}d\omega(z)\\
&\lesssim  \ell(B)^\beta\|b\|_{\Lambda^\beta} \sum_{j=2}^\infty \int_{\sigma(\Gamma), \|\sigma(x)-z\|\sim 2^j{\mathfrak R}^k} \Big(\frac{{\mathfrak R}^k}{\|x-z\|}\Big)^{1-\beta} \frac{ |f_2(z)|}{\omega(B(\sigma(x), \|\sigma(x)-z\|))}d\omega(z)\\
&\lesssim  \ell(B)^\beta\|b\|_{\Lambda^\beta}M(f_2)(\sigma(x)).
\end{aligned}
\end{equation}
We turn to considering $J_{\sigma,2}(x,y)$. Applying H\"older's inequality and then the definition of $\Lambda^\beta$,
we obtain
\begin{equation}\label{eq 4.7}
\begin{aligned}
J_{\sigma,2}(x,y)&\lesssim 
    \sum_{j=2}^\infty \int_{\sigma(\Gamma),  \|\sigma(x)-z\|\sim 2^j{\mathfrak R}^k} 2^{-j}
          \frac{ |f_2(z)|}{\omega(B(\sigma(x),2^j{\mathfrak R}^k))}|b_{B(\sigma(x),{\mathfrak R}^k)}-b(z)| d\omega(z)\\
  &\lesssim 
    \sum_{j=2}^\infty \int_{\sigma(\Gamma),  \|\sigma(x)-z\|\sim 2^j{\mathfrak R}^k} 2^{-j}
          \frac{ |f_2(z)|}{\omega(B(\sigma(x),2^j{\mathfrak R}^k))}( 2^j{\mathfrak R}^k)^\beta\|b\|_{\Lambda^\beta}  d\omega(z)\\        
  &\lesssim  \ell(B)^\beta\|b\|_{\Lambda^\beta}M(f_2)(\sigma(x)).
\end{aligned}
\end{equation}
Thus, by \eqref{eq 4.6} and \eqref{eq 4.7}, we have
$$\frac1{\omega(B) \ell(B)^\beta}\int_B |g_{22}(y)-g_{22}(x)|d\omega(y)\le C\|b\|_{\Lambda^\beta}\sum_{\sigma\in G}M(f_2)(\sigma(x)).$$
Finally, we turn to estimate $g_{\sigma_j 2}$. To this end we note that for $z\in U_j$ and $y\in B$ we have
$$\|z-y\|\ge \|z-x\|-\|x-y\|\ge 5{\mathfrak R}^k-{\mathfrak R}^k=4{\mathfrak R}^k,$$
$$\|x-\sigma_j(x)\|\le \|x-y\|+\|z-y\|+\|z-\sigma_j(x)\|\le 6{\mathfrak R}^k+\|z-y\|\le\frac52\|z-y\|,$$
and then
\begin{align*}
\int_B |R_{{\rm D},\ell}(z,y)|d\omega(y)
&\lesssim \int_B  \frac{d(z,y)}{\|z-y\|}\frac1{\omega(B(z,d(z,y)))}d\omega(y)\\
&\lesssim \frac{{\mathfrak R}^k}{{\mathfrak R}^k+\|x-\sigma_j(x)\|}\int_B \frac{d(z,y)}{{\mathfrak R}^k}      \frac1{\omega(B(z,d(z,y)))}d\omega(y)\\
&\lesssim \frac{{\mathfrak R}^k}{{\mathfrak R}^k+\|x-\sigma_j(x)\|}\int_{\mathcal O(B(x, 16{\mathfrak R}^k))} \frac{d(z,y)}{{\mathfrak R}^k}      \frac1{\omega(B(z,d(z,y)))}d\omega(y)\\
&\lesssim \frac{{\mathfrak R}^k}{{\mathfrak R}^k+\|x-\sigma_j(x)\|}\sum_{j=-4}^\infty\int_{d(z,y)\sim 2^{-j}{\mathfrak R}^k} \frac{d(z,y)}{{\mathfrak R}^k}      \frac1{\omega(B(z,d(z,y)))}d\omega(y)\\
&\lesssim \frac{{\mathfrak R}^k}{{\mathfrak R}^k+\|x-\sigma_j(x)\|}\sum_{j=-4}^\infty 2^{-j}\int_{d(z,y)\sim 2^{-j}{\mathfrak R}^k}     \frac1{\omega(B(z,d(z,y)))}d\omega(y)\\
&\lesssim \frac{{\mathfrak R}^k}{{\mathfrak R}^k+\|x-\sigma_j(x)\|}.
\end{align*}
Hence the proof of \eqref{eq 4.3} gives
\begin{align*}
&\frac1{\omega(B) \ell(B)^\beta}\int_B |g_{\sigma_j 2}(y)-(g_{\sigma_j 2})_B|d\omega(y)\\
&\qquad\le \frac2{\omega(B) \ell(B)^\beta}\int_B |g_{\sigma_j 2}(y)|d\omega(y)\\
&\qquad\le \frac2{\omega(B) \ell(B)^\beta}\int_B\int_{U_j} |R_{{\rm D},\ell}(z,y)||b_B-b(z)||f_{\sigma_j}(z)| d\omega(z)d\omega(y)\\
&\qquad\lesssim \frac{{\mathfrak R}^k}{{\mathfrak R}^k+\|x-\sigma_j(x)\|}\Big(\frac{\|x-\sigma(x)\|}{{\mathfrak R}^k}+1\Big)^\beta
\|b\|_{\Lambda^\beta} \frac1{\omega(B)}\int_{U_j} |f_{\sigma_j}(z)|d\omega(z)\\
&\qquad\lesssim \|b\|_{\Lambda^\beta}M(f)(\sigma_j(x)).
\end{align*}
Putting these estimates together, we now take the supremum over all $k\in \mathbb Z$ and then take $L^p_{\omega}$-norm of both sides. Then we use Proposition \ref{defference}, Theorem \ref{thm 4.2}, the $L^p_\omega$-boundedness of the maximal functions $M_t$, $M_{ts}$ and $M$ (noting that $t<p$ and $ts<p$), and the fact that the measure $d\omega$ is $G$-invariant, 
to conclude that
$$\|[b,R_{{\rm D},\ell}]f\|_{{\dot F}^{\beta,\infty}_{p,{\rm D}}}\lesssim \|b\|_{\Lambda^\beta}\|f\|_{L^p_\omega}$$
from \eqref{eq decom}.

Conversely, we now prove 
$$ \|b\|_{\Lambda^\beta}\|f\|_{L^p_\omega}\lesssim \|[b,R_{{\rm D},\ell}]f\|_{{\dot F}^{\beta,\infty}_{p,{\rm D}}}.$$
Assume that the commutator $[b,R_{{\rm D},\ell}]$ is a bounded operator from $L^p(\mathbb {R}^N,\omega)$ to ${\dot F}^{\beta,\infty}_{p,{\rm D}}$.
 For $B=B(x_0,r)$ with $x_0\in \mathbb R^N$ and $r>0$, we consider
\begin{align*}
\frac 1{\ell(B)^{\beta}\omega(B)}\int_{B} |f(x)-f_B| d\omega(x) .
\end{align*}

Choose a $C^\infty$ function $a(x)$ with the following properties:
\begin{itemize}
\item [(1)] supp $a \subseteq B(x_0, 3r)$, $\int_{\mathbb R^N} a(x)\,d\omega(x)=0$, $a=1$ on $B$,
\item[(2)]  $\|a\|_{\dot F^{-\beta, 1}_{p', {\rm D}}}\le r^\beta\omega(B)^{1/p'}$.
\end{itemize}

We provide an explanation for the existence of such a function $a(x)$.
That is, if $a(x)$ satisfies the conditions in (1) above, then we have the norm estimate in (2).

Recall that $(\Bbb R^N, \|\cdot\|, \omega)$ is the space of homogeneous type. 
Let $\{S_k\}_{k\in \Bbb Z}$ be Coifman's approximation to the identity. 
Set $D_k=S_k-S_{k-1}$. 
Choosing $k_0\in \mathbb Z$ with ${\mathfrak R}^{-k_0}\sim r$, we obtain
\begin{align*}
&\Big\|\Big\{\sum_{k\in \mathbb Z}{\mathfrak R}^{-k\beta}|D_k(a)|\Big\}\Big\|_{L^{p'}_{\omega}}\\
&\le \Big\|\Big\{\sum_{k=k_0+1}^\infty{\mathfrak R}^{-k\beta}|D_k(a)|\Big\}\Big\|_{L^{p'}_{\omega}}
       + \Big\|\Big\{\sum_{k=-\infty}^{k_0}{\mathfrak R}^{-k\beta}|D_k(a)|\Big\}\Big\|_{L^{p'}_{\omega}}\\
 &=:I_1+I_2.
\end{align*}
Since $S_k(x,y)=0$ for $\|x-y\|>{\mathfrak R}^{4-k}$, we have that 
$D_k(x,y)$ is supported in $\|x-y|\|\leq {\mathfrak R}^{5-k}$.
Note also that  supp $a\subseteq I(x_0,3r)$, we have that 
$$D_k(a)(x) = \int_{I(x_0,3r)}D_k(x,y)a(y)d\omega(y)$$
is supported in $I(x_0,Cr)$ for some absolute positive constant $C$. 
Moreover, we also have
$$|D_k(a)(x)| \leq \int_{I(x_0,3r)}|D_k(x,y)||a(y)|d\omega(y) 
             \leq \int_{I(x_0,3r)}|D_k(x,y)|d\omega(y)\lesssim1.$$
Thus,
\begin{align*}
I_1     &\le \Bigg(\int_{I(x_0,Cr)}\bigg\{\sum_{k=k_0+1}^\infty{\mathfrak R}^{-k\beta}\big|D_k(a)(x)\big|\bigg\}^{p'}d\omega(x)\Bigg)^{1/p'}\\
&\lesssim \Bigg(\int_{I(x_0,Cr)}\bigg\{\sum_{k=k_0+1}^\infty{\mathfrak R}^{-k\beta}\bigg\}^{p'}d\omega(x)\Bigg)^{1/p'}\\
&\leq C r^\beta \omega(I)^{1/p'}.
\end{align*}

To estimate $I_2$, we use the vanishing condition of $a(x)$ and smooth condition of $D_k$ to get
\begin{align*}
I_2&\le  \sum_{k=-\infty}^{k_0}\Big\|{\mathfrak R}^{-k\beta}|D_k(a)|\Big\|_{L^{p'}_{\omega}}\\
    &\le \sum_{k=-\infty}^{k_0}\bigg(\int_{I(x_0,C{\mathfrak R}^{-k})}\bigg|{\mathfrak R}^{-k\beta}\int_{I(x_0,3r)} \Big(D_k(x,y)-D_k(x,x_0)\Big)a(y)d\omega(y)\bigg|^{p'} d\omega(x)\Bigg)^{1\over p'}\\
     &\lesssim \sum_{k=-\infty}^{k_0}\bigg(\int_{I(x_0,C{\mathfrak R}^{-k})}\bigg|{\mathfrak R}^{-k\beta}\int_{I(x_0,3r)} \frac{|y-x_0|{\mathfrak R}^k}{\omega(I(y,{\mathfrak R}^{-k}))}d\omega(y)\bigg|^{p'} d\omega(x)\Bigg)^{1\over p'}\\
      &\lesssim \sum_{k=-\infty}^{k_0}\bigg(\int_{I(x_0,C{\mathfrak R}^{-k})}\bigg|{\mathfrak R}^{-k\beta}\int_{I(x_0,3r)} \frac{r{\mathfrak R}^k}{\omega(I(x_0,{\mathfrak R}^{-k})}d\omega(y)\bigg|^{p'} d\omega(x)\Bigg)^{1\over p'}\\
  &\lesssim \sum_{k=-\infty}^{k_0} r{\mathfrak R}^{-k(\beta-1)}\Big(\frac{\omega(I(x_0,r))}{\omega(I(x_0,{\mathfrak R}^{-k})}\Big
      )^{1-1/p'}\omega(I(x_0,r))^{1/p'}\\
      &\lesssim \sum_{k=-\infty}^{k_0} r{\mathfrak R}^{k(1-\beta)}\omega(I(x_0,r))^{1/p'}\\
      &\lesssim r^\beta \omega(I)^{1/p'}.
\end{align*}

We now continue with our proof.

\begin{defn}\label{mfb-N}
Let $f$ be finite almost everywhere on $\mathbb R^N$. For $B\subseteq \mathbb R^N$ 
with $\omega(B)<\infty$, we define a median value $m_f(B)$ of $f$ over $B$ to be a real number
satisfying
$$\omega(\{x\in B : f(x)>m_f(B)\})\le \frac12\omega(B)\qquad\mbox{and}\qquad 
\omega(\{x\in B : f(x)<m_f(B)\})\le \frac12\omega(B).$$
\end{defn}

Note that
\begin{align*}
[b, R_{{\rm D},\ell}]f(x)&=b(x)R_{{\rm D},\ell}f(x)-R_{{\rm D},\ell}(bf)(x)\\
 &=\int_{\mathbb R^N}(b(x)-b(y))R_{{\rm D},\ell}f(y)d\omega(y),
\end{align*}
where $$R_{{\rm D},\ell}(x,y)=-c\int_0^\infty \frac{y_\ell-x_\ell}t h_t(x,y)\frac{dt}{\sqrt t}.$$
We choose $\widetilde B=B(\tilde x_0, r)$ such that $y_\ell-x_\ell\ge r$ and $\|x-y\|\sim r$ for $x\in B$  and $y\in \widetilde B$. Set
 $$E_1:=\{y\in \widetilde B : b(y)\le m_b(\widetilde B)\}\qquad\mbox{and}\qquad 
    E_2:=\{y\in \widetilde B : b(y)> m_b(\widetilde B)\}.$$
Moreover, we define
$$B_1:=\{y\in B : b(y)\ge m_b(\widetilde B)\}\qquad\mbox{and}\qquad 
    B_2:=\{y\in B : b(y)< m_b(\widetilde B)\}.$$
Then by Definition \ref{mfb-N}, 
we see that $\omega(E_i)\ge \frac12\omega(\widetilde B), i=1,2.$ 
Moreover, for $(x,y)\in B_i\times E_i, i=1,2,$
\begin{align*}
|b(x)-b(y)|&=|b(x)-m_b(\tilde B)+m_b(\widetilde B)-b(y)|\\
   &=|b(x)-m_b(\widetilde B)|+|m_b(\widetilde B)-b(y)|\ge |b(x)-m_b(\widetilde B)|.
\end{align*}
Hence, we have the following facts,
\begin{equation}\label{eq 3.13}
\begin{aligned}
& \mbox{(i) } B=B_1\cup B_2, \widetilde B= E_1\cup E_2 
    \mbox{ and }\omega(E_i)\ge \frac12\omega(\widetilde B), i=1,2;\\
& \mbox{(ii) } b(x)-b(y) \mbox{ does not change sign for all }(x,y)\in B_i\times E_i, i=1,2; \\
&\mbox{(iii) } |b(x)-m_b(\tilde B)|\le |b(x)-b(y)| \mbox{ for all }(x,y)\in B_i\times E_i, i=1,2.
\end{aligned}
\end{equation}

We have that, for $(x,y)\in B_i\times E_i, i=1,2$, 
$$|R_{{\rm D},\ell}(x,y)|\ge \frac 1{\omega(B(\tilde x_0,r))}.$$
Let $f_i=\chi_{E_i}, i=1,2.$ 
Then the facts \eqref{eq 3.13} give
\begin{align*}
&\frac1{\ell(B)^{\beta}\omega(B)}\sum_{i=1}^2\bigg|\int_B [b, R_{{\rm D},\ell}](f_i)(x) a(x) d\omega(x)\bigg|\\
 &\ge \frac1{\ell(B)^{\beta}\omega(B)}\sum_{i=1}^2\bigg|\int_{B_i} [b, R_{{\rm D},\ell}](f_i)(x) a(x) d\omega(x)\bigg| \\
 &= \frac1{\ell(B)^{\beta}\omega(B)}\sum_{i=1}^2\int_{B_i}\int_{E_i} |b(x)-b(y)||R_{{\rm D},\ell}(x,y)| |a(x)| d\omega(y)d\omega(x)\\
 &\gtrsim  \frac1{\ell(B)^{\beta}\omega(B)}\sum_{i=1}^2\int_{B_i}|b(x)-m_b(\widetilde B)|\frac 1{\omega(B(\tilde x_
     0,r))}\int_{E_i} d\omega(y)d\omega(x)\\
  &\gtrsim  \frac1{\ell(B)^{\beta}\omega(B)}\sum_{i=1}^2\int_{B_i}|b(x)-m_b(\widetilde B)|d\omega(x)\\
 &\gtrsim \frac 1{\ell(B)^{\beta}\omega(B)}\int_B |b(x)-b_B| d\omega(x).
\end{align*}

On the other hand, from duality and the boundedness of $[b, R_{{\rm D},\ell}]$,
we deduce that
\begin{align*}
&\frac1{\ell(B)^{\beta}\omega(B)}\sum_{i=1}^2\bigg|\int_B [b, R_{{\rm D},\ell}](f_i)(x)a(x)d\omega(x)\bigg|\\
&\lesssim \frac1{\ell(B)^{\beta}\omega(B)}\sum_{i=1}^2  \|[b, R_{{\rm D},\ell}]f_i\|_{ \dot{F}_{p, {\rm D}}^{\beta,\infty}}
 \|a\|_{ \dot{F}_{p', {\rm D}}^{-\beta,1}} \\
&\lesssim \frac1{\ell(B)^{\beta}\omega(B)}\sum_{i=1}^2  \|[b, R_{{\rm D},\ell}]\|_{L^p_\omega\to \dot{F}_{p,{\rm D}}^{\beta,\infty}}\|f_i\|_{L^p_\omega} \|a\|_{ \dot{F}_{p' {\rm D}}^{-\beta,1}} \\
&\lesssim \frac1{\ell(B)^{\beta}\omega(B)} \|[b, R_{{\rm D},\ell}]\|_{L^p_\omega\to \dot{F}_{p, {\rm D}}^{\beta,\infty}}\omega(\widetilde B)^{1/p}r^\beta\omega(B)^{1/p'}\\
&\lesssim \|[b, R_{{\rm D},\ell}]\|_{L^p_\omega\to \dot{F}_{p,{\rm D}}^{\beta,\infty}}.
\end{align*}
Therefore, we have
$$\frac 1{\ell(B)^{\beta}\omega(B)}\int_B |b(x)-b_B| d\omega(x)\lesssim \|[b, R_{{\rm D},\ell}]\|_{L^p_\omega\to \dot{F}_{p,{\rm D}}^{\beta,\infty}}.$$ 
The proof of Theorem \ref{thm commutator Riesz} is complete.
\end{proof}

To prove Theorem \ref{main 2}, we need to estimate the kernel of fractional operator $\triangle_{\rm D}^{-\alpha/2}$.

\begin{lem}\label {lem 4.5}
Let $K_\alpha(x,y)$ be the kernel of fractional operator $\triangle_{\rm D}^{-\alpha/2}$.
Then we have the following estimates:
\begin{itemize}
\item[(a)] $\displaystyle |K_\alpha(x,y)|\lesssim   \frac{d(x,y)^{2+\alpha}}{\|x-y\|^2}\frac1{\omega(B(x,d(x,y)))}$
   for $d(x,y)\ne 0$;
\item[(b)]  $\displaystyle |K_\alpha(x,y)-K_\alpha(x,y')|\lesssim \frac{\|y-y'\|}{\|x-y\|^2}d(x,y)^{1+\alpha}\frac1{\omega(B(x,d(x,y)))}$ if $\|y-y'\|<\frac12 d(x,y)$;
\item[(c)]  $\displaystyle |K_\alpha(x,y)-K_\alpha(x',y)|\lesssim \frac{\|x-x'\|}{\|x-y\|^2}d(x,y)^{1+\alpha}\frac1{\omega(B(x,d(x,y)))}$ if $\|x-x'\|<\frac12 d(x,y)$.
\end{itemize}
\end{lem}

\begin{proof}
Note that
$$ \triangle_{\rm D}^{-\alpha/2}f(x) = {1\over \Gamma(\alpha/2)} \int_0^\infty {  e^{t\triangle_{\rm D}}}(f)(x) {dt\over t^{1-\alpha/2}}$$
and
$$e^{t\triangle_D}f(x)=\int_{\Bbb R^N} h_t(x,y)f(y)d\omega(y),$$
where $h_t(x,y)$ is the Dunkl-heat kernel.
We write the fractional operator as follows:
$$\triangle_{\rm D}^{-\alpha/2} f(x)= \int_{\mathbb R^N} K_\alpha(x,y)f(y)d\omega(y).$$

To estimate the kernel $K_\alpha(x,y),$ we recall the following estimates
for the Dunkl-heat kernel given in \cite[Theorem 3.1]{DH}

\begin{itemize}
    \item[(a)] There are constants $C,c>0$ such that
    $$|h_t(x,y)|\leqslant C\frac1{V(x,y,\sqrt t)}\bigg(1+\frac{\|x-y\|}{\sqrt t}\bigg)^{-2}e^{-cd(x,y)^2/t},$$
    for every $t>0$ and for every $x,y\in \mathbb R^N$.
    \item[(b)] There are constants $C,c>0$ such that
    $$|h_t(x,y)-h(x,y')|\leqslant C\bigg(\frac{\|y-y'\|}{\sqrt t}\bigg)\frac1{V(x,y,\sqrt t)}\bigg(1+\frac{\|x-y\|}{\sqrt t}\bigg)^{-2}e^{-cd(x,y)^2/t},$$
    for every $t>0$ and for every $x,y,y'\in \mathbb R^N$ such that $\|y-y'\|<\sqrt t$.
\end{itemize}
We now estimate the kernel $R_j(x,y)$ as follows.
\begin{align*}
|K_\alpha(x,y)|&\lesssim \int_0^\infty \frac1{V(x,y,\sqrt t)}\frac t{\|x-y\|^2}e^{-cd(x,y)^2/t}\frac{dt}{t^{1-\alpha/2}} \\
&\leqslant \frac1{\|x-y\|^2}\bigg(\int_0^{d(x,y)^2} + \int_{d(x,y)^2}^\infty\bigg) \frac1{V(x,y,\sqrt t)}e^{-cd(x,y)^2/t}t^\alpha dt\\
&=: I_1 +I_2.
\end{align*}
For $t\leqslant d(x,y)^2$, by using the doubling condition we have
that
$$\omega(B(x,d(x,y)))\lesssim \Big(\frac {d(x,y)}{\sqrt
t}\Big)^{\mathbf N}\omega(B(x,\sqrt t))$$ and hence
\begin{equation}\label{eq4.7}
V(x,y,\sqrt t)^{-1}\lesssim \frac1{\omega(B(x,\sqrt t))}\lesssim
\Big(\frac {d(x,y)}{\sqrt t}\Big)^{\mathbf
N}\frac1{\omega(B(x,d(x,y)))}.
\end{equation}
We obtain
\begin{align*}
I_1&\lesssim   \frac1{\|x-y\|^2}\frac1{\omega(B(x,d(x,y)))} \int_0^{d(x,y)^2}\Big(\frac {d(x,y)}{\sqrt t}\Big)^{\mathbf N}e^{-cd(x,y)^2/t} t^{\alpha/2} dt\\
&\lesssim   \frac1{\|x-y\|^2}\frac1{\omega(B(x,d(x,y)))} \int_0^{d(x,y)^2}\frac {d(x,y)^{\mathbf N}}{t^{\frac{\mathbf N}2}}\Big(\frac{t}{d(x,y)^2}\Big)^{\frac{\mathbf N}2} t^{\alpha/2} dt\\
&\lesssim  \frac{d(x,y)^{2+\alpha}}{\|x-y\|^2}\frac1{\omega(B(x,d(x,y)))}.
\end{align*}
It is clear that for $t\geqslant d(x,y)^2$, we get
\begin{align*}
I_2&\lesssim \frac1{\|x-y\|^2}\int_{d(x,y)^2}^\infty \frac1{V(x,y,d(x,y))}\frac {d(x,y)^{4+\alpha}}{t^{2+\alpha/2}}t^{\alpha/2} dt\\
&\lesssim  \frac{d(x,y)^{2+\alpha}}{\|x-y\|^2}\frac1{\omega(B(x,d(x,y)))}.
\end{align*}
To see the smoothness estimates, we write
\begin{align*}
|K_\alpha(x,y)-K_\alpha(x,y')|
&\leqslant c \int_0^\infty |h_t(x,y)-h_t(x,y')|\frac{dt}{t^{1-\alpha/2}}  \\
&\leqslant c\bigg(\int_0^{d(x,y)^2} +
\int_{d(x,y)^2}^\infty\bigg)
|h_t(x,y)-h_t(x,y')|\frac{dt}{t^{1-\alpha/2}} \\
&=: I\!I_1 +I\!I_2.
\end{align*}
Since $\|y-y'\|<\frac12 d(x,y)$, we have $d(x,y')\le \frac32d(x,y)$.
If $\|y-y'\|\ge \sqrt t$, we use the above condition (a) give
$$|h_t(x,y)-h_t(x,y')|\lesssim \Big(\frac {\|y-y'\|}{\sqrt t}\Big)\frac1{V(x,y,\sqrt t)}\bigg(1+\frac{\|x-y\|}{\sqrt t}\bigg)^{-2}e^{-cd(x,y)^2/t}. $$
If $\|y-y'\|< \sqrt t$, we use the above condition (a) give
$$|h_t(x,y)-h_t(x,y')|\lesssim \Big(\frac {\|y-y'\|}{\sqrt t}\Big)\frac1{V(x,y,\sqrt t)}\bigg(1+\frac{\|x-y\|}{\sqrt t}\bigg)^{-2}e^{-cd(x,y)^2/t}. $$
Hence, \eqref{eq4.7} gives
\begin{align*}
I\!I_1&\lesssim  \frac{\|y-y'\|}{\|x-y\|^2}\int_0^{d(x,y)^2}
\frac1{V(x,y,\sqrt t)}
e^{-cd(x,y)^2/t}\frac{dt}{t^{\frac12-\alpha}}\\
&\lesssim  \frac{\|y-y'\|}{\|x-y\|^2}\frac1{\omega(B(x,d(x,y)))}\int_0^{d(x,y)^2} \Big(\frac {d(x,y)}{\sqrt
t}\Big)^{\bf N}\Big(\frac t{d(x,y)^2}\Big)^{\frac {\bf N}2+\frac12-\frac\alpha2} \frac{dt}{t^{\frac{1-\alpha}2}}\\
&\lesssim \frac{\|y-y'\|}{\|x-y\|^2}d(x,y)^{1+\alpha}\frac1{\omega(B(x,d(x,y)))}.
\end{align*}

To estimate $I\!I_2$, the reverse doubling condition gives
\begin{align*}
I\!I_2&\lesssim  \frac{\|y-y'\|}{\|x-y\|^2}\int_{d(x,y)^2}^\infty
\frac1{V(x,y,\sqrt t)}
e^{-cd(x,y)^2/t}\frac{dt}{t^{\frac{1-\alpha}2}}\\
&\lesssim \frac{\|y-y'\|}{\|x-y\|^2}\frac1{\omega(B(x,d(x,y)))}\int_{d(x,y)^2}^\infty \Big(\frac {d(x,y)}{\sqrt
t}\Big)^N\Big(\frac t{d(x,y)^2}\Big)^{\frac N2-\frac32-\frac\alpha2} \frac{dt}{t^{\frac{1-\alpha}2}}\\
&\lesssim \frac{\|y-y'\|}{\|x-y\|^2}d(x,y)^{1+\alpha}\frac1{\omega(B(x,d(x,y)))}.
\end{align*}
The estimate of the smoothness in the $x$ variable is similar.  
\end{proof}

\begin{thm}\label{lpq bounded}
Let $0<\alpha< {\bf N}$. If $f\in L^p(\Bbb R^N, \omega), 1<p<q<\infty$ with $1/q=1/p-\alpha/{\bf N}$, 
then $$\|\triangle_{\rm D}^{-\alpha/2}f\|_{L^q_\omega} \lesssim \|f\|_{L^p_\omega} .$$
\end{thm}
\begin{proof}
From Lemma \ref{lem 4.5}, we have 
\begin{align*} 
|\triangle_{\rm D}^{-\alpha/2} f(x)|&= \bigg|\int_{\mathbb R^N} K_\alpha(x,y)f(y)d\omega(y)\bigg|\\
&\leq\int_{d(x,y)\leq s}  \frac{d(x,y)^{2+\alpha}}{\|x-y\|^2}\frac1{\omega(B(x,d(x,y)))}|f(y)|d\omega(y)\\
&\quad+\int_{d(x,y)>s}  \frac{d(x,y)^{2+\alpha}}{\|x-y\|^2}\frac1{\omega(B(x,d(x,y)))}|f(y)|d\omega(y)\\
&=:J_1+J_2.
\end{align*}
For $J_1$, by decomposing the set $\{{d(x,y)\leq s}\}$ into annuli and by noting that $s<d(x,y)\leq \|x-y\|$, we have
\begin{align*} 
J_1&\leq\sum_{j=0}^\infty\int_{2^{-j-1}s< d(x,y)\leq 2^{-j}s}  \frac{d(x,y)^{2+\alpha}}{\|x-y\|^2}\frac1{\omega(B(x,d(x,y)))}|f(y)|d\omega(y)\\
&\leq\sum_{j=0}^\infty\int_{2^{-j-1}s< d(x,y)\leq 2^{-j}s}  \frac{ (2^{-j}s)^{2+\alpha}}{(2^{-j-1}s)^2}\frac1{\omega(B(x,2^{-j-1}s))}|f(y)|d\omega(y)\\
&\leq C\sum_{j=0}^\infty(2^{-j}s)^{\alpha} \frac1{\omega(B(x,2^{-j}s))}\int_{d(x,y)\leq 2^{-j}s}  |f(y)|d\omega(y)\\
&\leq Cs^\alpha M(f)(x),
\end{align*}
where $M$ is the Hardy--Littlewood maximal function on the space of homogeneous type $(\mathbb R^N, \|\cdot\|, \omega)$.

For $J_2$, by noting that $s<d(x,y)\leq \|x-y\|$, we have
\begin{align*} 
J_2&\leq \bigg(\int_{d(x,y)>s}  \frac{d(x,y)^{(2+\alpha)p'}}{\|x-y\|^{2p'}}\frac1{\omega(B(x,d(x,y)))^{p'}}d\omega(y) \bigg)^{1\over p'}\|f\|_{L^p_\omega}\\
&\lesssim \bigg(\int_{d(x,y)>s}  d(x,y)^{\alpha p'}\frac1{\omega(B(y,d(x,y)))^{p'}}d\omega(y) \bigg)^{1\over p'}\|f\|_{L^p_\omega}\\
&\lesssim \bigg(\int_{d(x,y)>s} \frac{d(x,y)^{\alpha p'}}{d(x,y)^{N p'} \prod\limits_{\alpha\in R}\big(|\langle\alpha,y\rangle|+d(x,y)\big)^{\kappa(\alpha)p'}} \prod_{\alpha\in R}\big(|\langle\alpha,y\rangle|^{\kappa(\alpha)}dy \bigg)^{1\over p'}\|f\|_{L^p_\omega}\\
&\lesssim \bigg(\int_{d(x,y)>s}  \frac1{d(x,y)^{N p'-\alpha p'+\sum\limits_{\alpha\in R}\kappa(\alpha)(p'-1)} } dy \bigg)^{1\over p'}\|f\|_{L^p_\omega}\\
&\lesssim \bigg(\int_{t>s}  \frac1{t^{{\bf N}( p'-1)-\alpha p'+1}} dt \bigg)^{1\over p'}\|f\|_{L^p_\omega}\\
&\lesssim s^{\frac {\bf N}{p'}-{\bf N}+\alpha}\|f\|_{L^p_\omega}.
\end{align*}
Hence 
$$|\triangle_{\rm D}^{-\alpha/2} f(x)|\lesssim s^\alpha M(f)(x)+s^{\frac {\bf N}{p'}-{\bf N}+\alpha}\|f\|_{L^p_\omega}.$$
Choosing $s=\Big(\frac{\|f\|_{L^p_\omega}}{Mf(x)}\Big)^{\frac p{\bf N}}$, we obtain
$$|\triangle_{\rm D}^{-\alpha/2} f(x)|\lesssim \|f\|_{L^p_\omega}^{\frac{\alpha p}{\bf N}}Mf(x)^{1-\frac{\alpha p}{\bf N}}.$$
Note that $({\bf N}-\alpha p)q=p{\bf N}$.
$$\|\triangle_{\rm D}^{-\alpha/2} f\|_{L^q_\omega}\lesssim  \|f\|_{L^p_\omega}^{\frac{\alpha p}{\bf N}}
\|Mf\|^{1-\frac{\alpha p}{\bf N}}_{L^p_\omega}\lesssim \|f\|_{L^p_\omega}^{\frac{\alpha p}{\bf N}}
\|f\|^{1-\frac{\alpha p}{\bf N}}_{L^p_\omega}=\|f\|_{L^p_\omega}.$$
The proof of Theorem \ref{lpq bounded} is complete.
\end{proof}

Suppose for each cube $Q$ we have a function $h^Q$, defined on this cube. Let $\beta\ge 0$ and define
$$h_\beta(x)=\sup_{Q\ni x}\frac1{\omega(Q)^{1+\frac \beta{\bf N}}}
                       \int_Q |h^Q(y)|d\omega(y),\qquad x\in \mathbb R^N,$$
where the supremum is taken over all the cubes $Q\subseteq \mathbb R^N$ such that $x\in Q$.
To prove Theorem  \ref{main 2}, we need the following lemma.

\begin{lem}\label{pq lemma}
Let $0\le \beta<\alpha<\infty$ and $1<p<q<\infty, 1/p-1/q=(\alpha-\beta)/{\bf N}$. Suppose for each cube $Q$ we have a function $h^Q$, defined on this cube.
Then, 
\begin{equation}\label{lemma 5.1}
\|h_\beta\|_{L^q_\omega} \lesssim \|h_\alpha\|_{L^p_\omega},
\end{equation}
where the implied constant depends only on $p,q,\alpha$ and $\omega$.
\end{lem}

\begin{proof}
Let $0<r<\frac{\bf N}{\alpha-\beta}$ and let $Q$ be a cube in $\mathbb R^N$. We have that
\begin{align*}
\frac1{\omega(Q)^{1+\frac \beta{\bf N}}}
                       \int_Q |h^Q|d\omega
&\le \omega(Q)^{\frac{\alpha-\beta}{\bf  N}}\inf_{u\in Q} h_\alpha(u) \\
&\le \omega(Q)^{\frac{\alpha-\beta}{\bf  N}-\frac1r}
       \bigg(\int_Q h_\alpha(u)^r d\omega(u)\bigg)^{1/r} \\
&=\bigg(\frac1{\omega(Q)^{1-\frac{r(\alpha-\beta)}{\bf  N}}}
         \int_Q h_\alpha(u)^r d\omega(u)\bigg)^{1/r} \\
&\lesssim \bigg(\int_Q \frac{h_\alpha(u)^r}{\omega(B(x,|x-u|))^{1-\frac{r(\alpha-\beta)}{\bf  N}}}
         d\omega(u)\bigg)^{1/r}, \quad x\in Q. 
\end{align*}
We define, for every $0<\gamma<1$, the fractional integral $I^\gamma$ by
$$I^\gamma(f)(x)=\int_{\mathbb R^N} \frac{f(u)}{\omega(B(x,|x-u|))^{1-\gamma}}d\omega(u).$$
Then, 
\begin{equation}\label{hbeta}
h_\beta(x)\lesssim \{I^\gamma[h_\alpha]^r(x)\}^{1/r},\qquad x\in \mathbb R^N.
\end{equation}
It follows from \cite[Theorem 1]{Pan} that the operator $I^\gamma$ maps $L^{\tilde p}$ boundedly into $L^{\tilde q}$ whenever $1<\tilde p<\tilde q$ and $\frac1{\tilde p}=\frac1{\tilde q}+\gamma$.
Let $\tilde p=p/r$ and $\tilde q=q/r$ with $r<p$ and $\gamma=\frac{r(\alpha-\beta)}{\bf  N}<1$ as above. Then with $g=I^\gamma[(h_\alpha)^r]$, we have from \eqref{hbeta}
\begin{align*}
\|h_\beta\|_{L^q_\omega}
&\lesssim \|g^{1/r}\|_{L^q_\omega}
   =\|g\|^{1/r}_{L^{\tilde q}_\omega}\\
&\lesssim  \|(h_\alpha)^r\|^{1/r}_{L^{\tilde p}_\omega}\\
&=\|h_\alpha\|_{L^p_\omega},
\end{align*}
which is \eqref{lemma 5.1}. The proof is complete.
\end{proof}

We now prove Theorem \ref{main 2}.

\begin{proof}[Proof of Theorem \ref{main 2}]

Fix $x\in \mathbb R^N, k\in \mathbb Z$. Set $B=B(x, {\mathfrak R}^k)$ and  $f\in  L^p_\omega$ is a  compactly supported function.  Let $f_1, f_2$ and $f_{\sigma_j}$ be the decomposition of $f$ in \eqref{eq decom}. 
Observe that $[b,\triangle_{\rm D}^{-\alpha/2}]f=[b-b_B,\triangle_{\rm D}^{-\alpha/2}]f$. For $y\in \mathbb R^N$ we set
\begin{align*} 
&[b,\triangle_{\rm D}^{-\alpha/2}]f_1(y)=(b(y)-b_B)\triangle_{\rm D}^{-\alpha/2}f_1(y)+\triangle_{\rm D}^{-\alpha/2}((b_B-b)f_1)(y) =: h_{11}(y)+h_{12}(y),\\
&[b,\triangle_{\rm D}^{-\alpha/2}]f_2(y)=(b(y)-b_B)\triangle_{\rm D}^{-\alpha/2}f_2(y)+\triangle_{\rm D}^{-\alpha/2}((b_B-b)f_2)(y) =: h_{21}(y)+h_{22}(y),\\
&[b,\triangle_{\rm D}^{-\alpha/2}]f_{\sigma_j}(y)=(b(y)-b_B)\triangle_{\rm D}^{-\alpha/2}f_{\sigma_j}(y)+\triangle_{\rm D}^{-\alpha/2}((b_B-b)f_{\sigma_j})(y) =: h_{\sigma_j1}(y)+h_{\sigma_j2}(y).
\end{align*}     
We choose the index $t$ such that $1<t<p$ and let $t'$ be the conjugate of the index $t$. H\"older's inequality gives
\begin{align*}
&\frac1{\omega(B) \ell(B)^\beta}\int_B |h_{11}(y)-(h_{11})_B|d\omega(y)\le \frac2{\omega(B) \ell(B)^\beta}\int_B |h_{11}(y)|d\omega(y)\\
&\qquad\le \frac2{\omega(B) \ell(B)^\beta}\int_B |(b(y)-b_B)\triangle_{\rm D}^{-\alpha/2}f_1(y)|d\omega(y)\\
&\qquad\le \frac2{\ell(B)^\beta} \Big(\frac1{\omega(B)} \int_B |b(y)-b_B|^{t'} d\omega(y)\Big)^{\frac1{t'}} 
           \Big(\frac1{\omega(B)}\int_B |\triangle_{\rm D}^{-\alpha/2}f_1(y)|^{t}d\omega(y)\Big)^{\frac 1t}\\ 
&\qquad\le C\|b\|_{\Lambda^\beta}M_t(\triangle_{\rm D}^{-\alpha/2}f_1)(x),
\end{align*}
Thus, we use Theorems \ref{thm 4.2} and \ref{lpq bounded} to obtain
\begin{equation}\label{h11}
 \|h_{11}\|_{\dot F_{q,{\rm D}}^{\beta,\infty}} \lesssim \|b\|_{\Lambda^\beta}\|M_t(\triangle_{\rm D}^{-\alpha/2}f_1)\|_{L^q_\omega}\lesssim \|b\|_{\Lambda^\beta}\|\triangle_{\rm D}^{-\alpha/2}f_1\|_{L^q_\omega}\lesssim  \|b\|_{\Lambda^\beta}\|f_1\|_{L^p_\omega}.
\end{equation}
Similarly, we also have
\begin{equation}\label{h21}
 \|h_{21}\|_{\dot F_{q,{\rm D}}^{\beta,\infty}}+\sum_{j=1}^{|G|-1} \|h_{\sigma_j1}\|_{\dot F_{q,{\rm D}}^{\beta,\infty}} \lesssim \|b\|_{\Lambda^\beta}\Big(\|f_2\|_{L^p_\omega}+ \sum_{j=1}^{|G|-1} \|f_{\sigma_j}\|_{L^p_\omega}\Big).
 \end{equation}
 To deal with $h_{12}$, Theorems \ref{thm 4.2} gives
 \begin{align*}
 \|h_{12}\|_{\dot F_{q,{\rm D}}^{\beta,\infty}} 
&\sim \Big\|\sup_{k\in \mathbb Z} \frac1{\omega(B(\cdot, {\mathfrak R}^k)) \ell(B(\cdot, {\mathfrak R}^k))^\beta}\int_{B(\cdot, {\mathfrak R}^k)} |h_{12}(y)-(h_{12})_{B(\cdot, {\mathfrak R}^k)}|d\omega(y)\Big\|_{L^q_\omega}\\
&\lesssim \Big\|\sup_{k\in \mathbb Z}\frac1{\ell(B(\cdot, {\mathfrak R}^k))^\beta\omega(B(\cdot, {\mathfrak R}^k))^{1+\frac{\alpha}{\bf N}}} \int_{B(\cdot, {\mathfrak R}^k)} |h_{12}(y)-(h_{12})_{B(\cdot, {\mathfrak R}^k)}|d\omega(y)\Big\|_{L^p_\omega},
\end{align*} 
where the last inequality follows from Lemma \ref{pq lemma} for 
$h^Q=\frac1{\ell(B(\cdot, {\mathfrak R}^k))^\beta}|h_{12}(y)-(h_{12})_{B(\cdot, {\mathfrak R}^k)}|$.
We choose $s, 1<s<p$, and $\bar s$ such that $\frac 1s-\frac 1{\bar s}=\frac{\alpha}{\bf N}$. 
By H\"older's inequality,
\begin{align*}
&\frac1{\omega(B)^{1+\frac{\alpha}{\bf N}} \ell(B)^\beta}\int_B |h_{12}(y)-(h_{12})_B|d\omega(y)\\
&\qquad\le \frac2{\omega(B)^{1+\frac{\alpha}{\bf N}} \ell(B)^\beta}\int_B |h_{12}(y)|d\omega(y)\\
&\qquad= \frac2{\omega(B)^{1+\frac{\alpha}{\bf N}} \ell(B)^\beta}\int_B |\triangle_{\rm D}^{-\alpha/2}((b-b_B)f_1)(y)|d\omega(y)\\
 &\qquad\lesssim\frac1{\omega(B)^{1+\frac{\alpha}{\bf N}} \ell(B)^\beta}\|\triangle_{\rm D}^{-\alpha/2}((b-b_B)f_1)\|_{L^{\bar s}_\omega}\ \omega(B)^{1-\frac1{\bar s}}\\
 &\qquad\lesssim \frac1{\omega(B)^{\frac 1s} \ell(B)^\beta}\|\triangle_{\rm D}^{-\alpha/2}((b-b_B)f_1)\|_{L^{\bar s}_\omega}.
\end{align*}
Then we use Theorem \ref{lpq bounded} to obtain 
$$
\frac1{\omega(B)^{1+\frac{\alpha}{\bf N}} \ell(B)^\beta}\int_B |\triangle_{\rm D}^{-\alpha/2}((b-b_B)f_1)(y)|d\omega(y)\lesssim \frac1{\omega(B)^{\frac1s} \ell(B)^\beta}\|(b-b_B)f_1\|_{L^{ s}_\omega}.
$$
Let $s_1>1$ with $s_1s<p$ and $s_1'$ be the conjugate of the index $s_1$,
we use H\"older's inequality and \eqref{norm-d} to obtain that
\begin{align*}
\frac1{\omega(B)}\|(b-b_B)f_1\|^s_{L^s_\omega}
&=\frac1{\omega(B)}\int_{5B} |b(y)-b_B|^s|f|^sd\omega(y) \\
&\le {  \Big(\frac1{\omega(B)}\int_{5B} |b(y)-b_B|^{ss_1'}d\omega(y) \Big)^{\frac1{s_1'}}}
         \Big(\frac1{\omega(B)}\int_{5B} |f|^{ss_1} d\omega(y)\Big)^{\frac1{s_1}}\\     
&\lesssim  \ell(B)^{\beta s}\|b\|^s_{\Lambda^\beta} M^s_{ss_1}(f)(x)         
\end{align*}
the $L^p_\omega$-boundedness of the maximal function $M_{ss_1}$ implies that 
\begin{equation}\label{h12}
 \|h_{12}\|_{\dot F_{q,{\rm D}}^{\beta,\infty}} 
 \lesssim \|b\|_{\Lambda^\beta}\|f\|_{L^p_\omega}.
\end{equation}
We turn to  analyze $h_{22}$. Observe that for $z\notin {\mathcal O}(5B)$ and $y\in B$ we have $\|x-y\|\le d(x,z)/2$.
Let $\Gamma$ be a fixed closed Weyl chamber such that $x\in \Gamma$, then Lemma \ref{lem 4.5} gives
\begin{align*}
|h_{22}(y)-h_{22}(x)|
&=\bigg|\int_{\mathbb R^N} (K_\alpha(y,z)-K_\alpha(x,z))(b(z)-b_B)f_2(z)d\omega(z) \bigg|\\
&\le \sum_{\sigma\in G}\int_{\sigma(\Gamma)} |K_\alpha(y,z)-K_\alpha(x,z)||b(z)-b_B||f_2(z)|d\omega(z)\\
&\lesssim \sum_{\sigma\in G}\int_{\sigma(\Gamma)} \frac{\|y-x\|}{\|x-z\|^2}d(x,z)^{1+\alpha}
          \frac1{\omega(B(x,d(x,z)))}|b(z)-b_B| |f_2(z)|d\omega(z).
\end{align*}
Note that 
$$\omega(B(x,r))\sim r^N \prod_{\alpha\in R}\big(|\langle\alpha,x\rangle|+r\big)^{\kappa(\alpha)}\ge r^{\bf N}.$$
By the same way as in the estimate $J_\sigma(x,y)$, we have
\begin{align*}
&\frac1{\omega(B)^{\frac{\alpha}{\bf N}} \ell(B)^\beta}|h_{22}(y)-h_{22}(x)|\\
& \lesssim \frac1{\omega(B)^{\frac{\alpha}{\bf N}} \ell(B)^\beta}\sum_{\sigma\in G}\int_{\sigma(\Gamma)} \frac{{\mathfrak R}^k}{\|x-z\|^2}
          \frac{d(x,z)^{1+\alpha}}{\omega(B(\sigma(x),d(x,z)))}|b_{B(\sigma(x),{\mathfrak R}^k)}-b_B| |f_2(z)|d\omega(z)\\
          &\quad+\frac1{\omega(B)^{\frac{\alpha}{\bf N}} \ell(B)^\beta}\sum_{\sigma\in G} \frac{{\mathfrak R}^k}{\|x-z\|^2}
          \frac{d(x,z)^{1+\alpha}}{\omega(B(\sigma(x),d(x,z)))}|b_{B(\sigma(x),{\mathfrak R}^k)}-b(z)| |f_2(z)|d\omega(z)\\
          &\lesssim \|b\|_{\Lambda^\beta}\sum_{\sigma\in G}M(f_2)(\sigma(x))\sum_{j=2}^\infty 2^{-j(1-\alpha-\beta)}\\
          &\lesssim \|b\|_{\Lambda^\beta}\sum_{\sigma\in G}M(f_2)(\sigma(x))
\end{align*}
for $\alpha+\beta<1$.
Thus, we use Proposition \ref{defference} and Theorem \ref{lpq bounded} to obtain
\begin{align*}
\|h_{22}\|_{\dot F_{q,{\rm D}}^{\beta,\infty}}
&{\lesssim  \Big\|\sup_{k\in \mathbb Z}\frac1{\ell(B(\cdot,{\mathfrak R}^k))^\beta \omega(B(\cdot,\mathfrak R^k))} \int_{B(\cdot,{\mathfrak R}^k)} |h_{22}(y)-h_{22}(\cdot)|d\omega(y)\Big\|_{L^q_\omega}}\\
&{\lesssim  \Big\|\sup_{k\in \mathbb Z}\frac1{\ell(B(\cdot,{\mathfrak R}^k))^\beta\omega(B(\cdot,{\mathfrak R}^k))^{1+\frac{\alpha}{\bf N}}} \int_{B(\cdot,{\mathfrak R}^k)} |h_{22}(y)-h_{22}(\cdot)|\Big\|_{L^p_\omega}}\\
&\lesssim  \Big\|\sup_{k\in \mathbb Z}\frac1{\ell(B(\cdot,{\mathfrak R}^k))^\beta\omega(B(\cdot,{\mathfrak R}^k))^{\frac{\alpha}{\bf N}}} \sup_{y\in B(\cdot,{\mathfrak R}^k)} |h_{22}(y)-h_{22}(\cdot)|\Big\|_{L^p_\omega}\\
&\lesssim   \|b\|_{\Lambda^\beta}{\sum_{\sigma\in G}}\|M(f_2)(\sigma(\cdot))\|_{L^p_\omega}\\
&\lesssim  \|b\|_{\Lambda^\beta}\|f\|_{L^p_\omega}
\end{align*}
for $\alpha+\beta<1$.
{
Note that Theorem \ref{thm 3.12} establishes the equivalence between the Dunkl--Triebel--Lizorkin spaces and the well-understood Triebel–Lizorkin spaces on spaces of homogeneous type. Note again $\omega(B(x,r))\sim r^N \prod_{\alpha\in R}\big(|\langle\alpha,x\rangle|+r\big)^{\kappa(\alpha)}\ge r^{\bf N}$ for all $0<r<\infty$. From this lower bound of the measure of the ball $B(x,r)$, by using  \cite[Theorem 5.1 (2) (d)]{HHHLP}, we have $\dot F_{u,{\rm D}}^{\beta_0,2}\hookrightarrow L^s_\omega\sim \dot F_{s,{\rm D}}^{0,2}$ for $1/u-1/s=\beta_0/{\bf N}$ and $u<s$.}
Since $\triangle_{\rm D}^{-\alpha/2}\triangle_{\rm D}^{-\beta/2}=\triangle_{\rm D}^{-{(\alpha+\beta)}/2}$,
we get rid of the assumption $\alpha+\beta<1$ by the method in \cite[Facts 1.9 and 1.10]{P} and hence
\begin{equation}\label{h22}
 \|h_{22}\|_{\dot F_{q,{\rm D}}^{\beta,\infty}} 
 \lesssim \|b\|_{\Lambda^\beta}\|f\|_{L^p_\omega}.
\end{equation}

Finally, we turn to estimate $h_{\sigma_j 2}$. 
To this end we use the same method in the estimate $g_{\sigma_j 2}$ and Lemma \ref{lem 4.5} to get
\begin{align*}
&\frac1{\omega(B)^{\frac\alpha{\bf N}}}\int_B |K_\alpha(z,y)|d\omega(y)\\
&\lesssim {\mathfrak R}^{-k\alpha}\int_B  \frac{d(z,y)^{2+\alpha}}{\|z-y\|^2}\frac1{\omega(B(z,d(z,y)))}d\omega(y)\\
&\lesssim \frac{{\mathfrak R}^k}{{\mathfrak R}^k+\|x-\sigma_j(x)\|}\sum_{j=-4}^\infty 2^{-j(1-\alpha)}\int_{d(z,y)\sim 2^{-j}{\mathfrak R}^k}     \frac1{\omega(B(z,d(z,y)))}d\omega(y)\\
&\lesssim \frac{{\mathfrak R}^k}{{\mathfrak R}^k+\|x-\sigma_j(x)\|}
\end{align*}
for $\alpha<1$ and hence
$$
\frac1{\omega(B)^{1+\frac\alpha{\bf N}} \ell(B)^\beta}\int_B |h_{\sigma_j 2}(y)-(h_{\sigma_j 2})_B|d\omega(y)
\lesssim \|b\|_{\Lambda^\beta}M(f)(\sigma_j(x)).
$$
Thus, we use Theorems \ref{thm 4.2} and \ref{lpq bounded} to obtain
\begin{align*}
\|h_{\sigma_j 2}\|_{\dot F_{q,{\rm D}}^{\beta,\infty}}
&\lesssim  \Big\|\sup_{k\in \mathbb Z}\frac1{\ell(B(\cdot,{\mathfrak R}^k))^\beta\omega(B(\cdot,{\mathfrak R}^k))} \int_{B(\cdot,{\mathfrak R}^k)} |h_{\sigma_j 2}(y)-(h_{\sigma_j 2})_{B(\cdot,{\mathfrak R}^k)}|d\omega(y)\Big\|_{L^q_\omega}\\
&\lesssim  \Big\|\sup_{k\in \mathbb Z}\frac1{\ell(B(\cdot,{\mathfrak R}^k))^\beta\omega(B(\cdot,{\mathfrak R}^k))^{1+\frac{\alpha}{\bf N}}} \int_{B(\cdot,{\mathfrak R}^k)} |h_{\sigma_j 2}(y)-(h_{\sigma_j 2})_{B(\cdot,{\mathfrak R}^k)}|d\omega(y)\Big\|_{L^p_\omega}\\
&\lesssim   \|b\|_{\Lambda^\beta}\|M(f_2)(\sigma(\cdot))\|_{L^p_\omega}\\
&\lesssim  \|b\|_{\Lambda^\beta}\|f\|_{L^p_\omega}
\end{align*}
for $\alpha<1$.
Since $\triangle_{\rm D}^{-\alpha/2}\triangle_{\rm D}^{-\beta/2}=\triangle_{\rm D}^{-{(\alpha+\beta)}/2}$,
we get rid of the assumption $\alpha<1$ by the method in \cite[Facts 1.9 and 1.10]{P} and hence
\begin{equation}\label{hsigma2}
 \|h_{\sigma_j2}\|_{\dot F_{q,{\rm D}}^{\beta,\infty}} 
 \lesssim \|b\|_{\Lambda^\beta}\|f\|_{L^p_\omega}.
\end{equation}
By \eqref{h11}-\eqref{hsigma2},  the proof is complete.
\end{proof}

From the method of Theorem \ref{main 2}, we prove Theorem \ref{main 1} immediately.

\begin{proof}[Proof of Theorem \ref{main 1}]
 It suffices to verify that
$$\|([b,\triangle_{\rm D}^{-\alpha/2}]f)^\sharp\|_{L^q_\omega}
\lesssim \|b\|_{{\rm BMO}_{{\rm Dunkl}}}\|f\|_{L^p_\omega},$$
where $g^\sharp$ denotes the sharp maximal function
$$g^\sharp(x)=\sup_{B\ni x}\inf_{c\in \mathbb C}\frac1{\omega(B)}\int_{B} |g(y)-c|d\omega(y).$$
Note that the ${\rm BMO}_{{\rm Dunkl}}$  space has the similar result to Lemma \ref{lem 4.1} (see \cite[Lemma 4.1]{DH3}), Theorem \ref{main 1} follows from the proof of Theorem \ref{main 2}.
\end{proof}

   \bigskip
\noindent   
{\bf Acknowledgement: } 
The authors would like to thank the referee for carefully reading and checking the paper and for helpful suggestions, which made the paper more accurate and readable.
   \bigskip
%%%%%%%%%%%%%%%%%%%%%%%%%%%%%%%%%%%%%%%%%%%%%%
%%%%%%%%%%%%%%
%%%%%%%%%%%%%%%%%%%%%%%%%%%%%%%%%%%%%%%%%%%%%%
%%%%%%%%%%%%%%%%%%%%%%%%%%%%%%%%%%%%%%%%%%%%%%%%
%%%%%%%%%%%%%%%%%%%%%%%%%%%%
%%%%%%%%%%%%%%%%%%%%%%%%%%%%%%%%%%%%%%%%%%%%%%%%

\end{document}